\documentclass[12pt]{article}
\usepackage{preprint-layout} % estetics
\usepackage{preprint-notation} % standard notation
\usepackage[]{preprint-tools} % tools while editing

\DeclareMathOperator{\supp}{supp}
\DeclareMathOperator{\diam}{diam}
\DeclareMathOperator{\Span}{span}
\title{On the Design of Locking Free Ghost Penalty Stabilization and the Relation to CutFEM with Discrete Extension} 
\usepackage{authblk}
\author{Erik Burman \mbox{  } Peter Hansbo \mbox{  } Mats G. Larson}
\date{\today}

\begin{document}

\maketitle

\begin{abstract}
In this note, we develop a new stabilization mechanism for cut finite element methods that generalizes previous 
approaches of ghost penalty type in two ways: (1) The quantity that is stabilized and  (2) The choice of elements 
that are connected in the stabilization. In particular, we can stabilize functionals of the discrete function such as 
finite element degrees of freedom.  We subsequently show that the kernel of our ghost penalty operator defines 
a finite element space based on discrete extensions in the spirit of those introduced in {\emph{Burman, E.; Hansbo, P. and Larson, M. G., CutFEM Based on Extended Finite Element Spaces, arXiv2101.10052, 2021,}} \cite{BurHanLar21a}.
\end{abstract}

\section{Introduction}

\paragraph{Contributions.}
Cut finite element methods are based on embedding a  computational domain into a background mesh that is not 
required to match the boundary leading to so-called cut elements at the boundary.  Adding stabilization terms, we can 
control the variation of the discrete functions close to the boundary, which allows us to prove stability,  condition number 
estimates, and optimal order a priori error estimates.  Alternatively, we may use a discrete extension operator and solve 
the problem in a subspace of the finite element space where the unstable degrees of freedom are eliminated in such a way that optimal order approximation bounds are retained. These two approaches have the same goal: to stabilize the method but appear very different at first glance. 

In this note, we show that the definition of stabilization terms,  added to the weak statement,  may be generalized in two 
ways: (1) The stabilized quantity may be some functional of the discrete function, for instance, finite element degrees freedom. This allows us to stabilize the unstable modes more precisely than standard approaches, which may be viewed as element-based. (2) The choice of elements that are connected. Typically,  face neighbors, or connected patches are used, but we may stabilize by connecting elements intersecting the boundary to an element within a distance proportional to the mesh parameter.  We show that the generalized stabilization form fits into the standard 
abstract requirements, and as a consequence, we obtain stable and optimal order convergent methods for second-order 
elliptic problems. \textcolor{black}{Comparing the results obtained herein with those of \cite{BurHanLar21a} and the companion paper to the present work \cite{BHLL22}, it is straightforward to extend the results to elliptic problems of higher order.}

We also show that for a robust design of the ghost penalty, one may let the stabilization parameter tend to infinity without introducing locking.  The limit corresponds to strong enforcement of certain algebraic constraints, which are identical to constraints implemented in specific extension operator frameworks.  This illustrates the very close connection between stabilization and extension approaches.

\paragraph{Earlier Work.}
The idea of extending polynomial approximation from the interior to the boundary to enhance the stability of a numerical scheme was first introduced in \cite{HR09} for a fictitious domain method. That a similar effect, with the additional control of the condition number of the system matrix, could be achieved using penalty terms was discussed in \cite{Bu10} and further developed in \cite{BH12, MLLR14a, MLLR14b, Pre18, LM19} in the context of CutFEM methods using Nitsche's method for the weak imposition of interface conditions and in \cite{BHL15} in the context of unfitted finite element approximation of surface PDEs. A parallel development has considered achieving stability by agglomerating boundary and bulk elements, effectively extending the polynomial approximation space from the inside up to the boundary. In the context of nonconforming FEM, this was first introduced for fictitious domain methods using discontinuous Galerkin methods in \cite{JL13} and then for hybrid high order methods for interface problems in \cite{BE18}. Agglomeration is straightforward when discontinuous functions are considered for approximation but more delicate if the spaces have to be conforming. The first approach to agglomeration using $C^0$ approximation spaces was proposed in \cite{HWX17}, using element merging and hanging nodes. The approach using an extension of shape functions was then discussed in a series of papers \cite{BadVer18, BMV18, NB21, badia2022robust}. A general framework for discrete extension operators was then proposed in \cite{BurHanLar21a}, allowing for higher conformity of the FEM spaces. In a similar spirit, robust extension operators for splines were recently introduced in \cite{BHLL22}. A weak
stabilization based on penalizing the difference of a general finite element function and its extended counterpart was proposed in \cite{BNV22}. The objective of the present work is to detail under what conditions the Ghost penalty term is robust and propose a penalty term acting directly on degrees of freedom, which naturally connects to the discrete extension operators introduced in \cite{BurHanLar21a}.

\paragraph{Outline.}
In section 2, we introduce the general framework for the design of Ghost penalty terms and show how it applies to some examples from the literature. We also present nodal stabilization and discuss its implementation. The notion of locking is introduced, and the design criteria necessary to avoid locking are given. Section 3 is devoted to the analysis of the methods. First, we consider what conditions must be satisfied by the penalty term. Then we show that under certain sufficient conditions on the localization of the couplings in the penalty term, locking does not occur even for very large stabilization parameters. The paper's main contribution is to show that nodal stabilization satisfies the conditions for a robust and accurate ghost penalty term. Finally, in section 4, the theory is illustrated in some numerical examples.
%The main purpose of the stabilization form is to gain control of $v \in V_h$ on $\Omega_h$ in terms of control on 
%$\Omega$ and the stabilization form,  more precisely we wish to achieve 
%\begin{align}\label{eq:stab-est}
%\boxed{
% \| \nabla^m v \|^2_{\Omega_h} \lesssim \| \nabla^m v \|^2_\Omega + \| v \|^2_{s_{h,m}}, \qquad m=0,1, \quad v \in V_h
%}
%\end{align}
%Since we are here focusing on continuous finite element spaces essentially used to discretize second-order problems 
%we consider $m=0,1$, but the same ideas may be extended to higher-order problems.  

\section{The Stabilization Framework}
We develop a general framework for stabilization that relaxes current approaches in two ways. First, we allow more general 
choices of how the stabilization connects elements, and second, we allow stabilization of more general quantities, including functionals. Typical choices for functionals are the degree of freedom that enable stabilization of individual degrees of freedom and obtaining a stabilization where the penalty parameter may tend to infinity without inducing locking.  

\subsection{The Mesh and Finite Element Spaces}

\begin{itemize} 

\item Let $\mcT_{h,0}$ be a quasiuniform mesh, with mesh parameter $h \in (0,h_0]$ 
consisting of closed elements $T$,  on a closed polygonal domain 
$\Omega_0 \subset \IR^d$. Let $V_{h,0}$ be a finite element space on $\mcT_{h,0}$,
\[
V_{h,0} := \{v_h \in C^0(\Omega_0) : v_h\vert_{T} \in \mathbb{P}_k(T) ,\quad \forall T \in \mcT_{h,0}\}
\]
where $\mathbb{P}_k(T)$ denotes the set of polynomials of degree less than or equal to $k$ on the simplex $T$. We here consider the setting of $C^0$ finite elements, but the discussion below is easy to extend to the case of smoother approximation spaces using the ideas from \cite{BurHanLar21a}.
\item Let $\Omega \subset \Omega_0$ be a given closed domain and let $\mcTh = \{ T \in \mcT_{h,0} : T \cap \Omega \neq \emptyset \}$ be the active mesh and let $\Omega_h = \cup_{T \in \mcTh} T$. Let $V_h  = V_{h,0} |_{\Omega_h}$ 
be the active finite element space.  Let $V_{h,T} = V_h |_T$ be the local finite element space obtained by considering the restriction of any $v_h \in V_h$ to the element $T$.

\item Let $\mcB_h = \{ \varphi_i : i \in I \}$ be the global finite element basis in $V_h$ indexed by the set $I$.  For each $T \in \mcTh$ let 
$\mcB_{h,T} = \{ \varphi_{i,T} = \varphi_i |_T : i \in I_T\}$ be the element finite element basis in $V_{h,T}$, with  
$I_T \subset I$ the indices such that $T \subset \supp(\varphi_i)$.   Let $\mcB_{h,T}^* = \{ \varphi_{i,T}^* : i \in I_T \}$ 
be the element degrees of freedoms such that $\varphi^*_{j,T}(\varphi_{i,T}) = \delta_{ij}$ and note that $\Span(\mcB_{h,T}) = V_{h,T}^*$.  Since we have a conforming finite element space we have, $\varphi^*_{i,T'}(w) = \varphi^*_{i,T''}(w)$ for 
any two elements  $T'$ and $T''$  in the support of $\varphi_i$ and $w \in V_h |_{\supp(\varphi_i)}$.

\end{itemize}
The $L^2$-inner product over some domain $X$ will be denoted by 
\begin{equation}
(v,w)_X := \int_X u w ~\mbox{d}X, \mbox{ with norm } \|v\|_X := (v,v)_X^{\frac12}.
\end{equation}
For symmetric positive semi-definite bilinear forms $s(v,w)$ the associated (semi-) norm will be denoted $\|v\|_s:= s(v,v)^{\frac12}$.

\subsection{Definition of Stabilization Forms}
\textcolor{black}{
This section will propose some generic design criteria for ghost penalty stabilizations. We will then show that several known stabilizations enter the proposed framework and that this abstract design leads to methods with the desired properties. For simplicity, we restrict the presentation to methods based on extensions of polynomials on an element. Another possibility is to consider a patch of elements and then first project to a global polynomial on the patch that is then extended. The below arguments can be straightforwardly extended to that case.
}
\begin{itemize}

\item Consider two elements $T_1$ and $T_2$ in $\mcT_h$ and let $v_i^e \in \mathbb{P}_k(\IR^d)$ be the canonical 
extension of $v_i \in V_{T_i}$.  We may then define the jump 
\begin{align}
[v]_{T_1,T_2} = v_1^e - v_2^e \in \mathbb{P}_k(\IR^d)
\end{align}

\item For a symmetric positive semidefinite bilinear form $b: \mathbb{P}_k(\IR^d) \times \mathbb{P}_k(\IR^d) \rightarrow \IR$ 
and a pair of elements $T_1,T_2\in \mcT_h$, we define the stabilization term associated with the triple 
$(b,T_1,T_2)$  by 
\begin{align}
s_{m,b,T_1,T_2}(v,w) = \tau h^{\alpha_m} b([v]_{T_1,T_2},[w]_{T_1,T_2})  
\end{align}
where $\tau>0$ and $\alpha_m$ are parameters. Here,  $\alpha_m$ is determined in such a way that the form 
stabilises the $H^m$ norm for $m=0,1$, and $\tau>0$ is a stabilization parameter that typically is determined by the stability analysis. 

\item Let  $\mcT_h = \mcT_h^L \cup \mcT_h^S$ be a partition of the elements, into two subsets called large and 
small elements, where 
\begin{equation}
\mcT_h^L = \{ T \in \mcT_h : | T \cap \Omega | \geq  \gamma |T| \},\qquad \mcT_h^S = \mcT_h \setminus \mcT_h^L
\end{equation}
with $\gamma \in (0,1]$.  

\item 
Let $S_h: \mcT^S_h \rightarrow \mcT^L_h$ be a mapping, which assigns a large element $S_h(T) \in \mcT_h^L$ 
to each small element $T\in \mcT_h^S$.  We will focus on stabilization terms 
of the form
\begin{align}\label{eq:sbtsh_def}
s_{m,b,T,S_h(T)}(v,w) = \tau h^{\alpha_m} b([v]_{T,S_h(T)}, [w]_{T,S_h(T)}), \qquad T \in \mcT_h^S
\end{align}

\item Given a set $\mcS$ of triples of the form $(b,T,S_h(T))$ we define 
\begin{align}\label{eq:sh-general}
\boxed{s_{h,m}(v,w) = \sum_{(b,T,S_h(T) \in \mcS} s_{m,b,T,S_h(T)}(v,w) }
\end{align}

\item  For $\psi_T \in V_{h,T}^*$ we define the stabilization form
\begin{equation}
b_{m,T}(p,q) = \tau h^{\alpha_m} \psi_T(p) \psi_T(q), \qquad p,q \in \mathbb{P}_k(\IR^d) 
\end{equation}
which gives
\begin{equation}\label{eq:sh-functional-indiv}
s_{m,b_T,T, S_h(T)} (v,w) = \tau h^{\alpha_m} \psi_T([v]_{T,S_h(T)}) \psi_T([w]_{T,S_h(T)})
\end{equation}
An important special case is $\psi = \varphi_{i,T}^*$ which enables control of degree of 
freedom $i$.  To define which nodes need stabilization we let   
\begin{equation}\label{eq:indesx-split}
I = I^S \cup I^L
\end{equation}
be a partition of the global index set $I$ into the indices $I^S$ such that the corresponding basis 
functions does not contain any large element in their support and the complement $I^L = I \setminus I^S$.  
For each $i \in I^S$ let $T_i$ be an element such that $T_i \subset \supp(\varphi_i)$,  and define
\begin{align}\label{eq:sh-functional}
\boxed{s_{h,m}(v,w) 
= \sum_{i\in I^S}   \tau h^{\alpha_m} \varphi^*_{i, T_i} ([v]_{T_i,S_h(T_i)}) \varphi^*_{i,T_i}([w]_{T_i,S_h(T_i)})}
\end{align}
This construction enables us to stabilize individual degrees of freedoms precisely, and we refer to it as 
nodal stabilization.

\item The semi-norm induced by the stabilization will be denoted
\begin{equation}
 \| v \|^2_{s_{h,m}} = s_h(v,v)
\end{equation}
where we recall that subscript $m$ will take the values $0$ or $1$ depending on the stabilization is designed 
to give stability in the $L^2$-norm or the $H^1$-norm.
\end{itemize}

\subsection{Implementation}

Let us for simplicity consider piecewise linear elements on a triangulation and the nodal stabilization form (\ref{eq:sh-functional}). For each 
$i \in I^S$ we pick $T_i \subset \supp(\varphi_i)$ and we let $S_h(T_i)$ be an element in $\mcT_h^L$ close to 
$T$.  For instance,  $S_h(T_i)$ can be the element in $\mcT_h^L$ closest to $T_i$.  The functional 
$\varphi^*_{i,T_i}(v)$ is simply the nodal value in node $i$, denoted by $x_i$,  i.e., 
\begin{equation}
\varphi^*_i(v) = v(x_i) = \hatv_i
\end{equation}
Thus the stabilizing term for node $i$ takes the form 
\begin{equation}
s_{m,b,T_i, S_h(T_i)} (v,w) = \tau h^{\alpha_m} ([v(x_i) ]_{T_i,S_h(T_i)}) ([w(x_i)]_{T_i,S_h(T_i)})
\end{equation}
where 
\begin{equation}
[v(x_i) ]_{T_i,S_h(T_i)} = v(x_i) - \sum_{j \in I_{S_h(T_i)}} v(x_j) \varphi_{j,S_h(T_i)}^e (x_i)
\end{equation}
and $\{\varphi_{j,S_h(T_i)}\}_{j \in I_{S_h(T_i)}}$ is the element basis on element $S_h(T_i)$.
In matrix form we get  
\begin{equation}
s_{m,b,T_i, S_h(T_i)} (v,w) =\tau h^{\alpha_m} \hatv \cdot \widehat{B}_i \cdot \hatw
\end{equation}
with 
\begin{equation}
\widehat{B}_i = \widehat{\omega}_i \otimes \widehat{\omega}_i 
\end{equation}
where $\widehat{\omega}_i$ is the vector 
\begin{align}
\widehat{\omega}_i =  e_i - \sum_{j \in I_{S_h(T)}}  \varphi_{j,S_h(T_i)}^e (x_i) e_j
\end{align}
with four non zero elements,  $\{e_j\}_{j\in I}$ is the canonical basis in $\IR^N$ with $N = |I|$,  and 
$\varphi_{j,T}^e$ is the canonical extension of the elements basis functions from $S_h(T_i)$ to $\IR^d$.  
For other standard finite elements spaces that satisfies the Ciarlet definition,  see \cite{BreSco},  we 
have the same implementation with the modification that the number of basis functions on the elements 
are different.  Thus the implementation is very simple.

\subsection{Examples of Stabilization Forms}\label{sec:examples}

Below we include a couple of examples of common stabilization forms to illustrate how they fit into 
the framework and to emphasize that the proposed stabilization forms (\ref{eq:sh-general}) and  
(\ref{eq:sh-functional}) is indeed a natural extension of previous terms.

\paragraph{Example 1. (See \cite{Bu10}).} We may fit the standard Ghost or Face penalty in the framework as follows.  First 
let $\mcF_h$ be the set of all internal faces in $\mcT_h$  that belong to an element $T$ that intersects the 
boundary.  For piecewise linears the stabilization term takes the form 
\begin{equation}\label{eq:face-penalty}
s_{h,m}(v,w) = \sum_{F \in \mcF_h} h^{3-2m} ([ \nabla_n v ],[\nabla w])_F 
\end{equation}
where $m=0,1$, and
\begin{align}
[\nabla_n v ] = \nabla_{n_1} v_1 + \nabla_{n_2} v_1 
\end{align}
with $T_1$ and $T_2$ the elements sharing face $F$ and $v_i = v|_{T_i}$.  To set this term into our 
framework we have for each face the ordered triple $(b_F, T_1,T_2)$ where 
\begin{equation}
b_F(p,q) = (\nabla_{n_1} p,  \nabla_{n_1} q)_F, \qquad p,q \in \mathbb{P}_k(\IR^d)
\end{equation}
Then taking $\alpha_m = 3 - 2m$ we get 
\begin{align}
s_{b,T_1,T_2}(v,w) &= \tau h^{3 - 2m}  b_F ( [ v ]_{T_1,T_2}, [ w ]_{T_1, T_2} ) 
\\
&= \tau h^{3 - 2m} (\nabla_{n_1} [v]_{T_1,T_2} \nabla_{n_1} [v]_{T_1,T_2} )_F
\\
&= \tau h^{3 - 2m} ( [\nabla_n v], [\nabla_n v] )_F
\end{align} 
 where we used the identity 
 \begin{equation}
 \nabla_{n_1} [ v ]_{T_1,T_2} = \nabla_{n_1} (v_1 - v_2 ) = \nabla_{n_1} v_1 + \nabla_{n_2} v_2  = [ \nabla_n v ]
 \end{equation}
which holds since $n_2 = - n_1$ on $F$.  

\paragraph{Example 2. See \cite{Pre18,LM19}.} Taking
\begin{align}
b_F(p,q) = (\nabla p, \nabla q)_{T_1 \cup T_2}, \qquad p,q \in \mathbb{P}_k(\IR^d) 
\end{align}
and $\alpha_1 = 0$ gives us the form 
\begin{equation}
s_{1,b_F,T_1,T_2}(v,w) = (\nabla [v]_{T_1,T_2}, \nabla [w ]_{T_1,T_2})_{T_1\cup T_2}
= (\nabla (v_1^e - v_2^e) , \nabla (w_1^e - w_2^e) )_{T_1\cup T_2}
\end{equation}
which can be used to control the $H^1$ seminorm. Alternatively using the $L^2$ product 
\begin{align}
b_F(p,q) = (p, q)_{T_1 \cup T_2} 
\end{align}
with $\alpha_m = -2m$, $m=0,1$, gives us the form 
\begin{equation}\label{eq:face-penalty-elem}
s_{m,b_F,T_1,T_2}(v,w) = \tau h^{-2m} ([v]_{T_1,T_2},  [w ]_{T_1,T_2})_{T_1 \cup T_2}
= \tau h^{-2} (v_1 - v_2, w_1 - w_2 )_{T_1 \cup T_2}
\end{equation}
which can be used to control the $L^2$ norm and $H^1$ norm for $m=0$ and $m=1$, 
respectively.

\paragraph{Example 3.} For each $T \in \mcT_h^S$ 
we consider the triple $(b_T, T, S_h(T))$ with 
\begin{equation}
b_{1,T}(p,q) = (\nabla p, \nabla q)_T, \qquad p,q \in \mathbb{P}_k(\IR^d) 
\end{equation}
leading to the stabilization form 
\begin{equation}\label{eq:gradstab}
s_{1,b_T,T, S_h(T)} (v,w) = \tau (\nabla [v]_{T,S_h(T)} , \nabla [w]_{T,S_h(T)})_T
\end{equation}
Several variants are possible, for instance, we may take the domain of integration to be 
$T\cup S_h(T)$ or $S_h(T)$,  we can use the $L^2$ inner product and $\alpha_m = -2m$, $m=0,1$,  
and we may let $S_h:\mcT_h^S \rightarrow \mcT_h^I$ where $\mcTh^I$ is the set of elements 
residing in $\Omega$, i.e.  $T\subset \Omega$.  The final stabilization term takes the form 
\begin{align}
s_{h,1}(v,w) = \sum_{T \in \mcT_h^S} s_{1,b_T,T, S_h(T)}(v,w) 
= \sum_{T \in \mcT_h^S}   \tau (\nabla [v]_{T,S_h(T)} , \nabla [w]_{T,S_h(T)})_T
\end{align}
This stabilization shows that we may use flexible pairs of elements $(T,S_h(T))$ to construct a 
stabilization not only using face neighbors as in standard Ghost penalty. If we now relax the control on 
element $T$ using instead a functional $\psi_T \in V_{h,T}^*$ we let 
\begin{equation}
b_{m,T}(p,q) = h^{d - 2m} \psi_T(p) \psi_T(q), \qquad p,q \in \mathbb{P}_k(\IR^d) 
\end{equation}
and
\begin{equation}
s_{m,b_T,T, S_h(T)} (v,w) = \tau h^{d-2m} \psi_T([v]_{T,S_h(T)}) \psi_T([w]_{T,S_h(T)})
\end{equation}
which gives (\ref{eq:sh-functional-indiv}).

\paragraph{Example 4. (Stabilized version of the approach from \cite{HR09,BPV20}).} Recalling that in order to establish the coercivity of Nitsche's method 
we need the inverse inequality 
\begin{align}\label{eq:inversetrace}
h \| \nabla_n v \|^2_{T \cap \partial \Omega} \lesssim \| \nabla v \|^2_T
\end{align}
for elements $T \in \mcT_h$ that intersect the boundary,  and we may therefore consider 
\begin{align}
b_T(p,q) = (\nabla_n p, \nabla_n q)_{T\cap \partial \Omega}, \qquad p,q \in \mathbb{P}_k(\IR^d) 
\end{align}
leading to the stabilization form 
\begin{equation}\label{eq:stab4}
s_{b_T,T, S_h(T)} (v,w) = \tau h (\nabla_n [v]_{T,S_h(T)} , \nabla_n [w]_{T,S_h(T)})_{T\cap \partial \Omega}
\end{equation}
We then note that 
\begin{align}
h\| \nabla_n v \|^2_{T \cap \partial \Omega} 
&\lesssim  
h \| \nabla_n (v|_{S_h(T)})^e \|^2_{T \cap \partial \Omega} + h  \| \nabla_n [v]_{T,S_h(T)}  \|^2_{T \cap \partial \Omega}
 \\
 &\lesssim 
 \| \nabla v \|^2_{S_h(T)} +\| v \|^2_{s_{b_T,T, S_h(T)}} 
\end{align}
We conclude that we have constructed a stabilization targeting the inverse 
inequality needed for coercivity. Although this approach leads to stable fictitious domain methods using Nitsche's method, it does not alleviate the ill-conditioning of the system matrix. Note that we omit the index $m$ in the notation in this and the following example since we target a quantity that is not the $L^2$ or $H^1$ norm.

\paragraph{Example 5.} As an alternative to the stabilization in Example 4, resulting in a method that also is well conditioned, we let $n_T$ be the constant 
$L^2$ projection of the normal at $T \cap \partial \Omega$ \textcolor{black}{(or the normal at any point on $T \cap \partial \Omega$)} and note that 
\begin{align}
h\|\nabla_n v \|^2_{T \cap \partial \Omega} 
&\lesssim 
 h\|\nabla_{n_T} v \|^2_{T \cap \partial \Omega} + h \| n - n_T \|^2_{L^\infty(T \cap \partial \Omega}) \| \nabla v \|^2_{T \cap \partial \Omega}
\\
 &\lesssim 
 h\|\nabla_{n_T} v \|^2_{T \cap \partial \Omega} + h^3 \| \nabla v \|^2_{T \cap \partial \Omega}
 \\
 &\lesssim 
 \|\nabla_{n_T} v \|^2_{T} + \| v \|^2_{T}
\end{align}
This estimate suggests defining
\begin{equation}
b_T(p,q) = (\nabla_{n_T} p, \nabla_{n_T} q)_T + (p, q)_T
\end{equation}
and 
\begin{equation}\label{eq:stab5}
s_{b_T,T, S_h(T)} (v,w) = \tau_1 (\nabla_{n_T} [v]_{T,S_h(T)} , \nabla_{n_T} [w]_{T,S_h(T)})_{T}
+ 
\tau_2 ([v]_{T,S_h(T)} ,  [w]_{T,S_h(T)})_{T}
\end{equation}
with parameters $\tau_1$ and $\tau_2$.  Here we get control of the $L^2(T)$ norm as well as the 
control of the normal derivative necessary for the inverse inequality (\ref{eq:inversetrace}) to hold. 

 The $L^2$ 
control in example 5 is needed to derive condition number estimates. More precisely the crucial stability property that allows to prove bounds on the condition number \cite{EG06, Bu10, BH12, BHL15} takes the form
\begin{align}\label{eq:stab-est}
\boxed{
 \| \nabla^m v \|^2_{\Omega_h} \lesssim \| \nabla^m v \|^2_\Omega + \| v \|^2_{s_{h,m}}, \qquad m=0,1, \quad v \in V_h
}
\end{align}
\textcolor{black}{where the semi-norm induced by the stabilization satisfies, $\| v \|^2_{s_{h,m}} \lesssim \| \nabla^m v \|^2_{\Omega_h}$ for all $v$ in the finite element space (see Lemma \ref{lem:genconsistency} below).}
This type of bound is known to hold for the stabilizations of Examples 1, 2, 3 and 5. In this paper we will prove the bound for the penalty operator on the form \eqref{eq:sh-functional}.

\subsection{Locking}

The stabilization provides additional control of the solution in the interface zone and enhances the approximation's accuracy. However, increasing the stabilization parameter $\tau$ 
may lead to growing constants $C(\tau)$ in error estimates while maintaining optimal convergence order for 
each fixed value of the parameter $\tau$.  This type of phenomenon is referred to as locking; for fixed $h$ and large $\tau$ the finite element space may not be sufficiently rich to satisfy the constraint imposed by the penalty operator while at the same time providing a good approximation. Usually, the choice of $\tau$ is straightforward, and locking is not a major problem in CutFEM; however, when systems of equations and nonlinearities are considered, a robust variant may be a safer choice. How pronounced 
the locking problem depends on the specific nature of the stabilization. For instance,  the face based stabilization 
forms (\ref{eq:face-penalty}) 
and (\ref{eq:face-penalty-elem}) leads to locking since in the limit when $\tau \rightarrow \infty$ all elements 
at the boundary will be coupled, and the only free function is a global polynomial.  The coupling becomes less pronounced when the stabilization is  
based on pairs $(T,S_h(T))$ of elements. Indeed we will prove that the stabilizations of the form \eqref{eq:sh-general} can be made locking free under suitable assumptions on the mappings $S_h$, both for the element-based stabilization and the nodal 
stabilization (\ref{eq:sh-functional}). The key to these results is to show that any function $u \in H^1(\Omega)$ admits an interpolant $\pi_h u \in V_h$ with optimal approximation properties such that $s_h(\pi_h u, v_h) = 0$ for all $v_h \in V_h$. This form of discrete strong consistency is satisfied when the kernel of $s_h$ is large enough. We note that the nodal stabilization is locking-free under very mild assumptions. In contrast, the (sufficient) conditions on the element-based stabilization appear to be more difficult to satisfy.

\section{Analysis}
In this section we establish the basic properties of the stabilization $s_h$ defined by (\ref{eq:sh-general}) and 
(\ref{eq:sh-functional}).  We start with the straightforward analysis of the general stabilization form (\ref{eq:sh-general}) 
and present a stability result and a consistency result.  Then we turn to nodal stabilization, which demands slightly more 
complex arguments. Here we first show a stability result in the natural norms. Then we show that the stabilization 
parameter $\tau$ may tend to infinity without introducing locking or loss of the order of convergence.  

\subsection{Properties of the General Stabilization}
We start by specifying some properties of the stabilization used for the analysis. 
 
\paragraph{Assumptions.} The following holds uniformly for all  triples $(b,T,S_h(T)) \in \mcS$: 
\begin{description}
\item[A1.]   There is a constant such that 
\begin{equation}\label{eq:A1}
h^{\alpha_m} \| p \|_{b}^2 \lesssim \| \nabla^m p \|_T^2, \qquad p \in \mathbb{P}_k(\IR^d)
\end{equation}
%the number of stabilization forms (or triples) associated with an element $T \in \mcT_h$ 
%is  uniformly bounded, and the number of elements $T \in \mcT_h^S$ that are mapped by 
%$S_h$ to the same element is uniformly bounded.

\item[A2.] There is a path $\mcP_h(T,S_h(T)) = \{T_j\in \mcT_h : j=1,\dots, n \}$ of face neighboring elements 
starting in $T_1 = T$ and ending in $T_n = S_h(T)$ with bounded length $n \lesssim 1$. 

\end{description}
\begin{rem}
It is straightforward to verify that the examples given in Section \ref{sec:examples} satisfy the assumptions A1 and A2.
\end{rem}

\begin{lem}[{\bf Stability}]\label{lem:genstab} If A1 and A2 hold then there is a constant such that 
\begin{equation}
\boxed{\sum_{(b,T,S_h(T)) \in \mcS}  h^{\alpha_m} \| v \|^2_{b} \lesssim \| \nabla^m v \|^2_\Omega + \| v \|^2_{s_{h,m}} }
\end{equation}
\end{lem} 
\begin{proof} Adding and subtracting $(v|_{S_h(T)})^e$,  using the definition \eqref{eq:sbtsh_def} of 
$s_{b,T,S_h(T)}$, and finally assumption A1, we directly get
\begin{align}
h^{\alpha_m} \| v \|^2_{b} & \lesssim h^{\alpha_m} \| (v|_{S_h(T)})^e \|^2_{b} 
+  h^{\alpha_m} \| [v]_{T, S_h(T)} \|^2_{b}
\\
& \lesssim \| \nabla^m (v|_{S_h(T)})^e \|^2_T 
+  \| v \|^2_{s_{b,T,S_h(T)}}
\\
& \lesssim \| \nabla^m v \|^2_{S_h(T)} +  \| v \|^2_{s_{b,T,S_h(T)}}
\end{align}
Summing over all stabilizing forms, and using the fact that there is a uniformly bounded number of stabilization forms associated with each element $T \in \mcT_h^S$, we obtain the desired estimate.
\end{proof}

In preparation for the next lemma, we establish the following Poincar\'e type 
estimate for the jump operator.

\begin{lem}\label{lem:chainpoincare}  Let $S_h$ satisfy assumption A2 and $T_0 \in \mcT_h$ an 
element such that $T_0 \cup T \cup S_h(T) \subset B_\delta$. Then there is a constant 
such that
\begin{equation}
\| [ v ]_{T,S_h(T)} \|_{T_0} \lesssim h \| \nabla v \|_{\mcP_h(T,S_h(T))}
\end{equation}
\end{lem}
\begin{proof} Using the notation $T_1 = T$, $T_2 = S_h(T)$, and $v|_{T_i} = v_i$ for $i=1,2$, we get by 
adding and subtracting a constant function $w \in P_0(\IR^d)$, 
\begin{align}
\| [ v ]_{T_1,T_2} \|_{T_0}
&\lesssim  \| v_1^e - w  \|_{T_0} +   \|  w  - v_2^e \|_{T_0}
\\
&\lesssim  \| v_1^e - w  \|_{T_1} +   \|  w  - v_2^e \|_{T_2}
\\
&=  \| v - w  \|_{T_1} +   \|  w  - v  \|_{T_2}
\\
&\lesssim  \| v - w  \|_{\mcP_h(T_1, T_2)} 
\\ \label{eq:technical-poincare}
&\lesssim h \| \nabla v \|_{\mcP_h(T_1, T_2)} 
\end{align}
Here we used the stability 
\begin{equation}
\| p \|_{T'} \lesssim \| p \|_T, \qquad p \in \mathbb{P}_k(\IR^d)
\end{equation}
which holds when $T'\cup T \subset B_\delta$ for some $\delta \lesssim h$.  To verify the final estimate (\ref{eq:technical-poincare}) there is by assumption A2 a path $\mcP_h(T_1, T_2) = \{ T^j\}_{j=1}^n$ of face neighboring elements  such that $T^1 = T_1$ and $T^n=T_2$. Then recursively applying, the Poincar\'e inequality 
\begin{align}
\| v \|^2_{T'} \lesssim  \| v \|^2_{T''} + h^2 \| \nabla v \|^2_{T' \cup T''} , \qquad v \in H^1(T' \cup T'')
\end{align}
for elements $T',T''$ sharing a face, we get 
\begin{align}
\| v - w \|^2_{\mcP_h(T_1, T_2) } \lesssim \| v - w \|^2_{T_1}  + h^2 \| \nabla v \|^2_{\mcP_h(T_1,T_2)}
\lesssim h^2 \| \nabla v \|^2_{\mcP_h(T_1,T_2)}
\end{align}
where in the final inequality we choose $w$ to be the $L^2$ projection on constants on the element $T_1$.
\end{proof}

We will see in the analysis presented below that it is natural to require that 
a weak consistency estimate, see (\ref{eq:stab-cons}),  holds. This bound 
ensures that the stabilization is not too strong and that the convergence rate 
in a finite element method is not negatively affected by the stabilization. 

To discuss the consistency of the method we start by defining an interpolation operator $\pi_h : L^2(\Omega) \rightarrow V_h$. To that end we introduce the local $L^2$-projection $P_T:L^2(T) \rightarrow \mathbb{P}_k(T)$ and recall 
that there is an universal extension operator 
$E: H^s(\Omega) \ni v \mapsto v^E \in H^s(\IR^d)$,  see \cite{stein70},  satisfying the stability 
\begin{align}\label{eq:ext-stability}
\| v^E \|_{H^s(\IR^d)} \lesssim \| v \|_{H^s(\Omega)}, \qquad s\geq 0
\end{align}
Next we define a Cl\'ement interpolation operator of form 
\begin{align}
\pi_{h,Cl} : L^2(\Omega_h) \ni v \mapsto \sum_{i \in I} \varphi_i^*(P_{T_i} v|_{T_i} ) \varphi_i \in V_h
\end{align}
where $T_i \in \mcT_h$ is an element in the support of $\varphi_i$. Then we define the interpolation operator by 
composing the Cl\'ement operator and the extension operator,
\begin{equation}\label{eq:interpol}
\pi_h: L^2(\Omega)\ni v\mapsto \pi_{h,Cl} E v \in V_h
\end{equation} 
By standard arguments we have the stability 
\begin{align}\label{eq:pihstab}
\|\pi h \|_T \lesssim \| v \|_{\mcT_h(T)}
\end{align}
where $\mcT_h(T)$ is the set of all elements that share a node with $T$ and the 
the error estimate
\begin{equation}\label{eq:stand_est}
\| v - \pi_h v \|_{H^m(\Omega_h)} \lesssim h^{k+1 -m } \| v \|_{H^{k+1}(\Omega)}, \qquad m=0,1
\end{equation}
\begin{lem}[{\bf Weak consistency general interpolants}]\label{lem:genconsistency} If A1 and A2 hold then there are constants such that 
\begin{equation}\label{eq:sh-scaling}
\boxed{\| v \|_{s_{h,m}} \lesssim {\tau^{1/2}} \| \nabla^m v \|_{\Omega_h}, \qquad v \in V_h}
\end{equation}
and
\begin{align}\label{eq:stab-cons}
\boxed{ \|  \pi_h v \|_{s_{h,m}} \lesssim {\tau^{1/2}} h^{k+1 - m} \| v \|_{H^{k+1}(\Omega)},  \qquad v \in H^{k+1}(\Omega) }
\end{align}
\end{lem}

\begin{proof} 
First,  we prove (\ref{eq:sh-scaling}).  We have for $v \in V_h$,
\begin{align}
\| v \|^2_{s_h} &= \sum_{(b,T,S_h(T)) \in \mcS} \tau h^{\alpha_m} \|[v]_{T,S_h(T)}\|^2_{b} 
\\
&\lesssim  \sum_{(b,T,S_h(T)) \in \mcS} \tau  \|\nabla^m [v]_{T,S_h(T)}\|^2_{T}  
\\
&\lesssim  \sum_{(b,T,S_h(T)) \in \mcS} \tau h^{-2m} \| [v]_{T,S_h(T)}\|^2_{T}  
\\
&\lesssim  \sum_{(b,T,S_h(T)) \in \mcS} \tau  h^{2(1-m)} \| \nabla v \|^2_{\mcP_h(T,S_h(T))}  
\\
&\lesssim \tau \| \nabla^m v \|^2_{\mcT_h}
\end{align}
where we used some inverse estimates and Lemma \ref{lem:chainpoincare}. 
For the weak consistency estimate let $v \in H^{k+1}(\Omega)$ and consider the form associated with the triple $(b,T,S_h(T))$.  
We have for each $w \in \mathbb{P}_k(\IR^d)$, 
\begin{align}
&\| \pi_h v \|_{s_{b,T,S_h(T)}} = \| \pi_h (v-w) \|_{s_{b,T,S_h(T)}} 
\lesssim h^{\alpha_m/2} \| [ \pi_h (v - w)]_{T,S_h(T)} \|_{b} 
\\
&\qquad 
\lesssim  \| \pi_h (v - w ) \|_{H^m(T \cup S_h(T))}
\lesssim  \| v - w \|_{H^m(\mcT_h(\mcP_h(T, S_h(T))))}
\lesssim h^{k+1 - m} \| v \|_{H^{k+1}(B_{\delta'})}
\end{align} 
where in the last inequality we choose $w$ according to the Bramble-Hilbert lemma, see \cite{BreSco},  and $B_{\delta'}$ is an 
open ball containing $\mcT_h(\mcP_h(T , S_h(T)))$, the set of all elements that share a node with an element in the chain $\mcP_h(T,S_h(T))$ that connects $T$ and $S_h(T)$, with $\delta' \sim h$.  Summing the contributions we obtain
\begin{align}
\|  \pi_h v \|^2_{s_{h,m}}  &= \sum_{(b,T,S_h(T)) \in \mcS} \| \pi_h v \|^2_{s_{b,T,S_h(T)}} 
\\
&\qquad \lesssim  
 \sum_{(b,T,S_h(T)) \in \mcS}  \tau  h^{2(k+1)} \| v \|^2_{H^{k+1-m}(B_{\delta'})} 
 \lesssim \tau h^{2(k+1-m)} \| v \|^2_{H^{k+1}(\Omega)}
\end{align}
where we used the fact that the set of balls $B_{\delta'}$ containing the chains $\mcP_h(T, S_h(T))$ have 
finite overlap, i.e, the number of balls an arbitrary point belongs to is bounded.
\end{proof}

To see that the method is locking free provided the coupling induced by the stabilization $s_h$ 
is sufficiently local we introduce the following partition of the elements,
\begin{equation}
{\widetilde{\mathcal{T}}}^L_h 
= \mcT_h^L \setminus S_h(\mcT_h^S)
= \{T \in \mathcal{T}_h^L: \not \exists T' \in \mathcal{T}_h^S \mbox{ such that } T = S_h(T')\}
\end{equation}
and 
\begin{equation}
\mathcal{T}^M_h = \{T^M :  T^M= S_h^{-1}(T) \cup T \mbox{ for some } T \in \mathcal{T}_h^L \setminus \widetilde{\mathcal{T}}^L_h\}
\end{equation}
Observe that $\mathcal{T}^M_h$ is the set of macro elements defined by the operator $S_h$. We further reduce this set by merging any two (or more) $T^M_1, T^M_2 \in \mathcal{T}^M_h$ for which ${T}^M_1 \cap {T}^M_2 \ne \emptyset$ into $\widetilde{T}^M = T^M_1 \cup T^M_2$. The resulting set of merged macro elements is denoted $\widetilde{\mathcal{T}}^M_h$. This means that we merge all macro elements that share a node into larger macroelements that are isolated from each other in the sense that there is a layer 
of elements that are not macro elements between them.  It follows by the definition that the union of all elements 
in $\widetilde{\mathcal{T}}^L_h$ and $\widetilde{\mathcal{T}}^M_h$ is $\Omega_h$. 

To each merged macro element $\widetilde{T}^M \in \widetilde{\mcT}^M_h$ we associate a ball $B_{\delta(\widetilde{T}^M),\widetilde{T}^M}$ of diameter $\delta(\widetilde{T}^M)$, such that $\widetilde{T}^M \subset B_{\delta(\widetilde{T}^M),\widetilde{T}^M}$. 

We next decompose the set of finite element basis functions into $\mathcal{B}_h^M$, the basis functions defining polynomials on the elements in the set $\widetilde{\mathcal{T}}^M_h$ and $\mathcal{B}_h^L = \mathcal{B}_h \setminus \mathcal{B}_h^M$, with associated index sets $I_M$ and $I_L$. Observe that since the elements in $\widetilde{\mathcal{T}}^M_h$ have disjoint boundaries each basis function is either in $\mathcal{B}_h^L$ and
 then its support has zero intersection with the elements in $\widetilde{\mathcal{T}}^M_h$, or attributed to a single element $\widetilde{T}^M$ in $\mathcal{B}_h^M$. 
 
We are now ready to define the modified interpolant
\begin{equation}\label{eq:interpolmerged}
\pi_h^M v := \sum_{i \in I_L} \varphi_{i}^*(\pi_h v) \varphi_i + \sum_{i \in I_M} \varphi_{i}^*(P_{B_{\delta,\widetilde{T}^M(i)}} v^E) \varphi_i 
\end{equation}
where $B_{\delta,\widetilde{T}^M(i)}$ denotes a ball associated to the unique $\widetilde{T}^M \in \widetilde{\mathcal{T}}^M$ of the basis function $\varphi_i$ such that $\widetilde{T}^M
\subset B_{\delta,\widetilde{T}^M(i)}$. Observe that it follows by the definition that $\pi_h^M v\vert_{\widetilde{T}^M} = P_{B_{\delta(\widetilde{T}^M),\widetilde{T}^M}} v$ for all $\widetilde{T}^M \in \widetilde{\mathcal{T}}^M_h$. As a consequence $[\pi_h^M v]_{T,S_h(T)} = 0$ for all $T \in \mathcal{T}^S_h$. We then have the following result.

\begin{lem}[{\bf Strong consistency}]\label{lem:strongconsistency}  The interpolant defined by (\ref{eq:interpolmerged}) satisfies
\begin{align}\label{eq:disc_strong_cons}
\boxed{ \|  \pi_h^M v \|_{s_{h,m}} = 0}
\end{align}
If there is a constant such that $\delta(\widetilde{T}^M) \lesssim h$ for all $ \widetilde{T}^M \in \widetilde{\mathcal{T}}^M_h$ and the set of balls 
$\{B_{\delta(\widetilde{T}^M),\widetilde{T}^M} : \widetilde{T}^M \in \widetilde{\mcT}_h^M \}$
have uniformly finite overlap. Then there is a constant such that 
\begin{equation}\label{eq:LF-approx}
\boxed{\| v - \pi_h^M v \|_{H^m(\Omega_h)} \lesssim h^{k+1 -m } \| v \|_{H^{k+1}(\Omega)}, \qquad m=0,1}
\end{equation}
\end{lem}
\begin{proof}
The first claim is true by construction, since for all $\widetilde{T}^M \in \widetilde{\mathcal{T}}_h^M$ $\pi_h^M v\vert_{\widetilde{\mathcal{T}}_h^M} = P_{B_{\delta(\widetilde{T}^M),\widetilde{T}^M}} v$ and therefore $[\pi_H^M v]_{T,S_h(T)} = 0$. To prove the error estimate \eqref{eq:LF-approx} we note that
\begin{equation}
\| v - \pi_h^M v \|_{H^m(\Omega_h)} \leq \| v - \pi_h v \|_{H^m(\Omega_h)}+\| \pi_h v - \pi_h^M v \|_{H^m(\Omega_h)}
\end{equation}
The first term of the right hand side is bounded by the standard estimate \eqref{eq:stand_est}.
For the second part we observe that
\begin{equation}
\pi_h v - \pi_h^M v = \sum_{i \in I_M} \varphi_{i}^*(B_{\delta(\widetilde{T}^M),\widetilde{T}^M(i)} v - \pi_h v) \varphi_i. 
\end{equation}
It then follows using an inverse inequality that
\begin{equation}
\| \pi_h v - \pi_h^M v \|_{H^m(\Omega_h)} \lesssim (h^{-m} + 1) \| \pi_h v - \pi_h^M v \|_{\Omega_h} \lesssim  (h^{-m} + 1) \| \pi_h v - \pi_h^M v \|_{\tilde \Omega_M}
\end{equation}
where $\tilde \Omega_M = \cup_{\widetilde{T}^M \in \widetilde{\mathcal{T}}^M_h} \widetilde{T}^M$. Applying once again the triangle inequality we see that
\begin{equation}
\| \pi_h v - \pi_h^M v \|_{\tilde \Omega_M} \leq \| \pi_h v - v \|_{\Omega_h}+ \|v - \pi_h^M v \|_{\tilde \Omega_M}.
\end{equation}
It only remains to bound the second term of the right hand side. It follows by the definition that
\begin{align}
\|v - \pi_h^M v \|_{\tilde \Omega_M}^2 &= \sum_{\tilde T_M \in \widetilde{\mathcal{T}}_M} \|v - \pi_h^M v \|_{\tilde T_M}^2 \leq \sum_{\tilde T_M \in \widetilde{\mathcal{T}}_M} \|v - \pi_h^M v \|_{B_{\delta(\widetilde{T}^M),\widetilde{T}^M}}^2 
\\
&\qquad 
\lesssim h^{2(k+1)} \sum_{\tilde T_M \in \widetilde{\mathcal{T}}_M} |v^E|^2_{H^{k+1}(B_{\delta(\widetilde{T}^M),\widetilde{T}^M})} \lesssim h^{2(k+1)} |v|^2_{H^{k+1}(\Omega)}  
\end{align}
where we used the finite overlap of the $B_{\delta(\widetilde{T}^M),\widetilde{T}^M}$ and the stability of the extension in the last inequality. We conclude by collecting the above bounds.
\end{proof}
We note that the assumptions on $\widetilde{\mathcal{T}}^M_h$ are quite strong. It is not obvious how to design an agglomeration map $S_h$ that ensures the uniform bound $\diam(\tilde  T^M) = O(h)$ for all $\tilde  T^M \in \widetilde{\mathcal{T}}^M_h$ as well as finite overlap of the associated set $\{B_{\delta(\widetilde{T}^M),\widetilde{T}^M} : \widetilde{T}^M \in \widetilde{\mcT}_h^M \}$ of balls.  We therefore now focus on what we call nodal-based stabilization, which does not need such strong assumptions to be locking free.

\subsection{Properties of Nodal Stabilization}

Here we specialize the analysis to the nodal stabilization defined by  (\ref{eq:sh-functional}) and 
we show that the critical stability bounds (\ref{eq:stab-est}) indeed hold and that we can construct an interpolation 
operator that satisfies strong consistency; more precisely, the interpolant of an $L^2$ function 
is in the kernel of the stabilization. 

Starting with the stability estimate we note that applying Lemma \ref{lem:genstab} we get 
\begin{align}
 \sum_{i\in I^S}   h^{d-2m} (\varphi^*_{T_i} (v))^2 \lesssim h^{-2m} \| v \|^2_\Omega + \| v \|^2_{s_h}
\end{align}
Thus the simple general analysis does not work directly in the case $m=1$, since 
we have the term  $h^{-2m} \| v \|^2_\Omega$ instead of $\| \nabla v \|_\Omega^2$. To handle that case we will need the following technical, but natural, assumption that extends A2.
\begin{description}
\item[A3.]
For elements $T_1, T_2$ in $\mcT_h^S$ contained in $\supp(\varphi_i)$ there is a 
path $\mcP_h(S_h(T_1), S_h(T_2))$ of uniformly bounded length consisting only of elements in 
$\mcT_h^L$. For $T_1\in \mcT_h^L$ and $T_2 \in \mcT_h^S$ contained in $\supp(\varphi_i)$ there is a 
path $\mcP_h(T_1, S_h(T_2))$ of uniformly bounded length consisting only of elements in 
$\mcT_h^L$.
\end{description}
\begin{lem}[{\bf Stability}]\label{lem:nodalstability}  Let $s_h$ be the nodal stabilization form defined by  (\ref{eq:sh-functional}).  Then 
there is a constant such that 
\begin{align}\label{eq:nodalstability}
\boxed{
 \| \nabla^m v \|^2_{\Omega_h} \lesssim \| \nabla^m v \|^2_\Omega + \| v \|^2_{s_{h,m}}, \qquad m=0,1, \quad v \in V_h
}
\end{align}

\eqref{eq:stab-est} holds.

\end{lem}
\begin{proof} We first recall that we have the element wise equivalence 
\begin{align}\label{eq:elem-eqv}
\| v \|^2_T \sim %\sum_{i \in I_T} h^{d_i} |\hatv_i|^2
h^d \| \hatv \|^2_{\IR^{N_T}}
\end{align}
see \cite{EG06},  where $\hatv_{i} = \varphi_{i,T}^*(v)$.  We then have the following estimate
\begin{align}
\| \nabla^m v \|^2_{\Omega_h} &=  \| \nabla^m v \|^2_{\mcT_h^L} + \| \nabla^m v \|^2_{\mcT_h^S}
\\
&\lesssim   \| \nabla^m v \|^2_{\mcT_h^L} + 
\sum_{T \in \mcT_h^S} \Big(   \| \nabla^m (v|_{S_h(T)})^e \|^2_{T}  
+ \| \nabla^m [v]_{T,S_h(T)} \|^2_{T}  \Big)
\\
&\lesssim  \| \nabla^m v \|^2_{\mcT_h^L} + 
\underbrace{\sum_{T \in \mcT_h^S}  \| \nabla^m v \|^2_{S_h(T)}}_{\leq \| \nabla^m v \|^2_{\mcT_h^L}}
+ \sum_{T \in \mcT_h^S} h^{-2m} \|  [v]_{T,S_h(T)} \|^2_{T} 
\\
&\lesssim  \| \nabla^m v \|^2_{\mcT_h^L} 
+ \underbrace{\sum_{T \in \mcT_h^S} \sum_{i \in I_T}  h^{d-2m}  |\varphi_{i,T}^* [v]_{T,S_h(T)}|^2 }_{\bigstar}
\\ \label{eq:technical-aa}
&\lesssim 
 \| \nabla^m v \|^2_{\mcT_h^L} + \| v \|^2_{s_{h,m}} 
\end{align}
where we for each element $T\in \mcT_h^S$ added and subtracted $(v|_{S_h(T)})^e, $ used the stability 
$\| \nabla^m (v|_{S_h(T)})^e \|_T \lesssim \| \nabla^m v \|_{S_h(T)}$ of polynomial extension,  the equivalence 
(\ref{eq:elem-eqv}),  and finally the estimate 
\begin{align}\label{eq:technical-a}
\bigstar \lesssim
 \| \nabla^m v \|^2_{\mcT_h^L} + \| v \|^2_{s_{h,m}} 
\end{align}
which we verify next.  To that end we will consider the global degrees of freedom in $I^S$ and $I^L$ 
separately.

\paragraph{Case 1: $\boldmath{i} \boldmath{\in} \boldmath{I}^{\boldmath{S}}$.} Consider a global degree of freedom $i \in I^S$ 
with $T_i \subset \supp(\varphi_i)$ 
the element assigned in the stabilization.  For each $T \in \mcT_h^S$ such that $T \subset \supp(\varphi_i)$,  we wish to express $\varphi^*_{i,T}[v]_{T,S_h(T)}$ in terms of the stabilized functional $\varphi^*_{i,T_i} [v]_{T_i,S_h(T)}$.  Adding and subtracting 
$(v|_{S_h(T_i)})^e$ we have the identity 
\begin{align}
[v]_{T,S_h(T)} = [v]_{T,S_h(T_i)} + [v]_{S_h(T_i),S_h(T)}
\end{align}
Here we observe that 
\begin{align}\label{eq:T2Ti}
\varphi^*_{i,T} [v]_{T,S_h(T_i)}  = \varphi^*_{i,T_i} [v]_{T_i,S_h(T_i)} 
\end{align}
since the finite element space is conforming. 
%and 
%\begin{equation}
%w = [v]_{T,S_h(T_i)} 1_T + [v]_{T_i,S_h(T_i)} 1_{T_i} \in V_h|_{T \cup T_i} 
%\end{equation}
Using these observations we have the estimate 
\begin{align}
h^{d-2m}  |\varphi_{i,T}^* [v]_{T,S_h(T)}|^2 
&
\lesssim 
h^{d-2m}  |\varphi_{i,T}^* [v]_{T,S_h(T_i)}|^2 + h^{d-2m}  |\varphi_{i,T}^* [v]_{S_h(T_i),S_h(T)}|^2 
\\
&
\lesssim 
h^{d-2m}  |\varphi_{i,T_i}^* [v]_{T_i,S_h(T_i)}|^2 + h^{-2m}  \|[v]_{S_h(T_i),S_h(T)}\|_T^2 
\\
&
\lesssim 
h^{d-2m}  |\varphi_{i,T_i}^* [v]_{T_i,S_h(T_i)}|^2 +  h^{2-2m} \|\nabla v \|^2_{\mcP_h(S_h(T_i) , S_h(T))}
\\
&
\lesssim
h^{d-2m}  |\varphi_{i,T_i}^* [v]_{T_i,S_h(T_i)}|^2 +   \|\nabla^m v \|^2_{\mcP_h(S_h(T_i), S_h(T))}
\end{align}
where we used identity (\ref{eq:T2Ti}) to replace $T$ by $T_i$, the boundedness of $\varphi_{i,T}^*$, Lemma \ref{lem:chainpoincare}, and finally an inverse 
estimate in the case $m=0$.

\paragraph{Case 2: $\boldmath{i} \boldmath{\in} \boldmath{I}^{\boldmath{L}}$.} 
For $i \in I^L$ we instead let $T_i$ be a large element in the support of $\varphi_i$, i.e. 
 $T_i \in \mcT_h^L$ and $T\subset \supp(\varphi_i)$, which gives,  using the same arguments 
 as in Case 1,
\begin{align}
h^{d-2m}  |\varphi_{i,T}^* [v]_{T,S_h(T)}|^2 
&
\lesssim 
h^{d-2m}  |\varphi_{i,T}^* [v]_{T,T_i}|^2 + h^{d-2m}  |\varphi_{i,T}^* [v]_{T_i,S_h(T)}|^2 
\\
&
\lesssim 
 h^{-2m}  \|[v]_{T_i,S_h(T)}\|_{T_i}^2 
\\
&
\lesssim 
 h^{2 - 2m} \|\nabla v \|^2_{\mcP_h(T_i, S_h(T))}
\\
&
\lesssim 
\| \nabla^m v \|^2_{\mcP_h(T_i, S_h(T))} 
\end{align}
%\begin{align}
%h^{d-2m}  |\varphi_{i,T}^* [v]_{T,S_h(T)}|^2 
%&
%\lesssim 
%h^{d-2m}  |\varphi_{i,T}^* [v]_{T,T_i}|^2 + h^{d-2m}  |\varphi_{i,T}^* [v]_{T_i,S_h(T)}|^2 
%\\
%&
%\lesssim 
%h^{d-2m}  |\varphi_{i,T_i}^* [v]_{T_i,S_h(T_i)}|^2 + h^{-2m}  \|[v]_{S_h(T_i),S_h(T)}\|_T^2 
%\\
%&
%\lesssim 
%h^{-2m}  \|[v]_{T_i,S_h(T_i)}\|^2_{T_i}   + h^{-2m}  \|[v]_{S_h(T_i),S_h(T)}\|_T^2 
%\\
%&
%\lesssim 
%h^{2 - 2m}  \| \nabla v \|^2_{\mcT_h(T_i, S_h(T_i))} + h^{2 - 2m} \|\nabla v \|^2_{\mcT_h(S_h(T_i), S_h(T))}
%\\
%&
%\lesssim 
%\| \nabla^m v \|^2_{\mcT_h(T_i, S_h(T_i))} + \|\nabla^m v \|^2_{\mcT_h(S_h(T_i), S_h(T))}
%\end{align}

\paragraph{Conclusion.} Using the estimates in Case 1 and 2 we obtain (\ref{eq:technical-a}) as follows
\begin{align}
\bigstar &= \sum_{T \in \mcT_h^S} \sum_{i \in I_T}  h^{d-2m}  |\varphi_{i,T}^* [v]_{T,S_h(T)}|^2
\\
&=  \sum_{T \in \mcT_h^S} \sum_{i \in I_T \cap I^S}  h^{d-2m}  |\varphi_{i,T}^* [v]_{T,S_h(T)}|^2
+ \sum_{T \in \mcT_h^S} \sum_{i \in I_T \cap I^L}  h^{d-2m}  |\varphi_{i,T}^* [v]_{T,S_h(T)}|^2
\\
&\lesssim \sum_{T \in \mcT_h^S} \sum_{i \in I_T \cap I^S}  h^{d-2m}  |\varphi_{i,T_i}^* [v]_{T_i,S_h(T_i)}|^2 
+  \|\nabla^m v \|^2_{\mcP_h(S_h(T_i), S_h(T))}
\\
&\qquad + \sum_{T \in \mcT_h^S} \sum_{i \in I_T \cap I^L}  \|\nabla^m v \|^2_{\mcP_h(T_i, S_h(T))}
\\
&\lesssim \| v \|^2_{s_{h,m}} + \| \nabla^m v \|^2_{\mcT_h^L}
\end{align}
which together with estimate (\ref{eq:technical-aa}) conclude the proof.

\end{proof}

Next we study the limit when the stabilization parameter becomes large we will see that for the the 
nodal stabilization (\ref{eq:sh-functional}) we retain optimal order approximation properties.  More precisely, 
we show that  there is an interpolation operator $\Pi_h:L^2(\Omega) \rightarrow V_h$ such that  
\begin{align}\label{eq:strong-consistency}
\boxed{ \| \Pi_h u \|_{s_h} = 0 }
\end{align}
which we may view as a strong version of (\ref{eq:stab-cons}),  and the optimal order interpolation estimate 
(\ref{eq:Pih-est}) holds. 

Next we define an interpolation operator $\widetilde{\Pi}_h : V_h \rightarrow V_h$, with the 
special property that only $v|_{\mcT_h^L}$ is used to determine the nodal values, 
\begin{align}\label{eq:tildePih-def}
V_h \ni v \mapsto \widetilde{\Pi}_h v = \sum_{i \in I^S} \varphi_{i,T_i}^*( (v|_{S_h(T_i)})^e ) \varphi_i  
+  \sum_{i \in I^L} \varphi_{i,T_i}^*(v) \varphi_i \in V_h
\end{align}
Finally,  we define the interpolation operator
\begin{align}\label{eq:Pih}
\boxed{\Pi_h: L^2(\Omega) \ni v \mapsto \widetilde{\Pi}_h \pi_h v \in V_h }
\end{align}
By construction the strong consistency (\ref{eq:strong-consistency}) holds.
 
\begin{lem}[\bf{Approximation}] There is a constant such that 
\begin{align}\label{eq:Pih-est}
\boxed{ \| v - \Pi_h v \|_{H^m(T)} \lesssim h^{k+1-m} \| v^E \|_{H^{k+1}(B_{\delta,T})} }
\end{align}
where $B_{\delta, T}$ is a ball containing $\mcT_h(T) \cup \mcT_h( S_h(T))$ with radius $\delta \sim h$ and we recall that $\mcT_h(T')$ is the set of elements that share a node with $T'$.  
\end{lem}
\begin{proof} Let $I_T = I_T^S \cup I_T^L$ be the indices to the basis functions $\varphi_i$ with 
$T \subset \supp(\varphi_i)$.   We first note that we have the $L^2$ stability 
\begin{align}
\| \Pi_h v \|_T \lesssim \| v^E \|_{B_{\delta,T}}
\end{align}
since 
\begin{align}
\| \Pi_h v \|^2_T &\lesssim \sum_{i \in I_T^S} h^d |\varphi_i^*(\pi_h v|_{S_h(T_i)})^e)|^2   
+ \sum_{i \in I_T^L} h^d |\varphi_i^*(\pi_h v)|^2  
\\
&\lesssim \sum_{i \in I_T^S}  \| (\pi_h v|_{S_h(T_i)})^e \|_{T_i}^2   + \sum_{i \in I_T^L}  \|\pi_h v\|_{T_i}^2  
\\
&\lesssim \sum_{i \in I_T^S}  \| \pi_h v \|_{S_h(T_i)}^2   + \sum_{i \in I_T^L}  \|\pi_h v\|_{T_i}^2  
\\
&\lesssim \sum_{i \in I_T^S}  \|  v^E  \|_{\mcT_h(S_h(T_i))}^2   + \sum_{i \in I_T^L}  \| v^E \|_{\mcT_h(T_i)}^2  
\\
&\lesssim \| v^E \|^2_{B_{\delta,T}}
\end{align}
where $T_i$ is a stabilized element for $i \in I_T^S$ and a large element for $i \in I_T^L$, and $B_{\delta,T}$ 
is a ball such that $\mcT_h(S_h(T_i)) \cup \mcT_h(T_i) \subset B_{\delta,T}$. We note that there is such a ball 
with radius $\delta \lesssim h$, due to shape regularity and assumption A2, for each $T$ and that the set of all 
such balls has uniformly bounded intersection. Furthermore, we used the stability 
of polynomial extension and the stability of the element $L^2$ and the stability $\| v \|_T = \|\pi_{h,Cl} v^E \|_{T_i} \lesssim \| v \|_{\mcT_h(T_i)}$. 

Next to prove (\ref{eq:Pih-est}) we first note that for  $w \in \mathbb{P}_k(B_{\delta,T})$ we have 
\begin{align}
(\pi_h w) |_T = w|_T 
\end{align}
which gives 
\begin{align}
\| v - \pi_h v \|_{H^m(T)} &\leq \| v - w \|_{H^m(T)} + \| w - \pi_h v \|_{H^m(T)} 
\\
& \leq \| v - w \|_{H^m(T)} +h^{-m} \| \pi_h ( w -v ) \|_{H^m(T)} 
\\
&\lesssim \| v - w \|_{H^m(T)} + h^{-m} \| w - v \|_{B_{\delta,T}}
\end{align}
Finally,  using the Bramble-Hilbert Lemma,  see \cite{BreSco},  we may choose $w \in \mathbb{P}_k(B_{\delta,T})$  such that 
\begin{equation}
\| v - \pi_h v \|_{H^m(T)} \lesssim h^{k+1-m} \| v \|_{H^{k+1}(B_{\delta,T})}
\end{equation}
as was to be shown.
\end{proof}

\subsection{Relation to Discrete Extension Operators}

We note that 
\begin{equation}
\widetilde{\Pi}_h : V_h \rightarrow V_h^E \subset V_h
\end{equation}
where $V_h^E = \text{Im}(\widetilde{\Pi}_h) = \widetilde{\Pi}_h V$ is a proper subspace of $V_h$ in fact 
\begin{equation}
\text{ker}(\widetilde{ \Pi}_h ) = \Span \{\varphi_i \in \mcB_h : i \in I^S \}
\end{equation} 
with dimension $\text{dim}(\text{ker}(\widetilde{\Pi}_h ) ) = |I^S|$ and as a consequence $\text{dim}(V_h^E) = |I^L|$. 
The operator $\widetilde{\Pi}_h$ restricted to $V_h$ is indeed identical to extension operators developed in \cite{BMV18a} 
and \cite{BurHanLar21a}.

\subsection{Application to CutFEM}

For completeness, we include an application of the stabilization forms to CutFEM.   We present two error estimates, one 
based on the weak consistency and one on the strong consistency provided by the operator $\Pi_h$. To that end let us consider the elliptic model problem 
\begin{equation}\label{eq:mod-prob}
-\Delta u  = f \quad \text{in $\Omega$}, \qquad u = 0  \quad \text{on $\partial \Omega$}
\end{equation}
and the cut finite element method: find $u_h \in V_h$ such that 
\begin{equation}\label{eq:fem}
A_h(u_h,v) = l_h(v) 
\end{equation}
where the forms are defined by
\begin{align}\label{eq:Ah}
A_h(v,w) &= a_h(v,w) + s_{h}(v,w)
\\
a_h(v,w) &= (\nabla v, \nabla w)_\Omega - (\nabla_n v, w)_{\partial \Omega} - (v, \nabla_n w)_{\partial \Omega} 
+ \beta h^{-1} (v, w)_{\partial \Omega}
\\
l_h(v) &= (f,v)_\Omega %- (g,\nabla_n v)_{\partial \Omega} + \beta h^{-1} (g, v)_{\partial \Omega}
\end{align}
where $\beta \in [\beta_0,\infty)$ for some $\beta_0>0$. Using partial integration we 
note that the exact solution to (\ref{eq:mod-prob}) satisfies  the unstabilized equation
\begin{align}
a_h(u,v) = l_h(v)\qquad \forall v \in V_h
\end{align}
In the following we shall account for the effect of the stabilization parameter $\tau$, which is hidden inside the stabilization form, and to make that dependence clear we from here on adopt the notation $s_{h,\tau}$. We also assume that there is a parameter $\tau_0>0$ such that 
$\tau \in [\tau_0,\infty)$. Define the Nitsche norms
\begin{align}
\tn v \tn_h^2 &= \| \nabla v \|^2_\Omega
+  \beta^{-1} h \| \nabla_n v \|^2_{\partial \Omega} + \beta h^{-1} \| v \|^2_{\partial \Omega}
\\
\tn v \tn_{h,\bigstar}^2 &=\tn v \tn_h^2 + \| v \|^2_{s_{h,\tau}}
\end{align}
Assuming that the stabilization form $s_{h,\tau}$ is defined in such a way that 
\begin{equation}\label{eq:nitsche-inverse}
h \|\nabla_n v \|^2_{\partial \Omega} %\lesssim \| \nabla v \|^2_{\mcT_h(\partial \Omega)}
\lesssim  \| \nabla v \|^2_\Omega + \|v \|^2_{s_{h,\tau_*}}
\end{equation}
for some fixed positive parameter $\tau_*$. Then the form $a_h$ is coercive on $V_h$ for sufficiently large $\beta\in [\beta_0, \infty)$,
\begin{align}\label{eq:nitsche-coer}
\tn v \tn_{h,\bigstar}^2 \lesssim A_h(v,v) \qquad v \in V_h
\end{align}
and $\tau \geq \tau_0 \gtrsim \beta^{-1}$. For the convenience of the reader we have included the proof of this result in Appendix \ref{app:coer}. We also have the continuity 
\begin{equation}
A_h(v,w) \lesssim \tn v \tn_{h,\bigstar} \tn w \tn_{h,\bigstar}, \qquad v,w \in V_h + H^{3/2+\epsilon}(\Omega)
\end{equation}
which follows directly from the Cauchy-Schwarz inequality.

\begin{rem} The inverse inequality (\ref{eq:nitsche-inverse}) holds for instance for the 
different stabilization forms in Examples 1-5 with $m=1$.
\end{rem}

\begin{prop} If the stabilization form $s_{h,\tau}$ is such that the stability estimate 
 (\ref{eq:stab-est})  and the consistency estimate (\ref{eq:stab-cons})  hold with $m=1$,  and 
 the stabilization form implies satisfaction of the inverse inequality (\ref{eq:nitsche-inverse}). Then there is a  constant such that for $\tau\geq \tau_0$,
\begin{equation}\label{eq:error_bound}
\tn u - u_h \tn_h \lesssim (1 + \tau^{1/2}) h^k \| u \|_{H^{k+1}(\Omega)}
\end{equation}
If in addition there exists an interpolant $\pi_h$ satisfying \eqref{eq:disc_strong_cons} and \eqref{eq:LF-approx}, then 
\begin{equation}\label{eq:error_bound-b}
\tn u - u_h \tn_h \lesssim  h^k \| u \|_{H^{k+1}(\Omega)}
\end{equation}
\end{prop}

\begin{proof} Splitting the error by adding and subtracting an interpolant we get
\begin{align}\label{eq:nit-a}
\tn u - u_h \tn_h \leq \tn u - \pi_hu \tn_h + \tn \pi_hu - u_h \tn_h 
\lesssim  h^k \| u \|_{H^{k+1}(\Omega)} +  \tn \pi_h u - u_h \tn_h 
\end{align}
For the second term we use the coercivity as follows
\begin{align}\label{eq:nit-b}
\tn \pi_h u - u_h \tn_h \leq \tn \pi_h u - u_h \tn_{h,\bigstar}
\lesssim \sup_{v \in V_h }\frac{A_h(\pi_h u - u_h,v)}{\tn v \tn_{h,\bigstar}}
\end{align}
where we have the identity 
\begin{align}
A_h(\pi_h u - u_h,v) &= A_h(\pi_h u, v ) - l_h(v) 
\\
&=  A_h(\pi_h u, v ) -  a_h(u,v) 
\\ \label{eq:pert_dev}
&= a_h( \pi_h u - u, v) + s_{h,\tau}(\pi_h u, v)
\end{align}
Estimating the right hand side gives
\begin{align}
|A_h(\pi_h u - u_h,v)| &\lesssim \tn  \pi_h u - u \tn_h \tn v \tn_h + \|\pi_h u\|_{s_{h,\tau}} \|v\|
_{s_{h,\tau}}
\\ \label{eq:nit-c}
&\lesssim ( \tn  \pi_h u - u \tn_h + \|\pi_h u\|_{s_{h,\tau}} )  \tn v \tn_{h,\bigstar}
\end{align}
for $v\in V_h$. Combining the bounds (\ref{eq:nit-a}), (\ref{eq:nit-b}), and (\ref{eq:nit-c}), we 
obtain 
\begin{align}
\tn u - u_h \tn_h &\lesssim \tn \pi_h u - u_h \tn_h  + \|\pi_h u\|_{s_{h,\tau}} 
\\ \label{eq:nit-d}
&\lesssim  h^{k} \| u \|_{H^{k+1}(\Omega)}  + \tau^{1/2}  h^{k} \| u \|_{H^{k+1}(\Omega)} 
\end{align}
where we used the interpolation estimate (\ref{eq:interpol}) and the weak consistency estimate (\ref{eq:stab-cons}). The second claim (\ref{eq:error_bound-b}) follows by noting that $\|\pi_h u\|_{s_{h,\tau}}$ in \eqref{eq:nit-d} is zero and hence $s_{h,\tau}$ does not contribute to the upper bound.
\end{proof}

\begin{rem} Note that the constants in the resulting estimate will depend on the choices of the stabilization parameter 
$\tau$  and the Nitsche penalty parameter $\beta$.  It follows from the proof of the coercivity estimate (\ref{eq:nitsche-coer}) that we may indeed take $\tau 
\sim \beta^{-1}$. Thus $\tau$ is a parameter of moderate size or even quite small in practice. Nevertheless, formulations that are less sensitive to parameters are, in general preferable.
\end{rem}

\begin{rem}
The interpolant $\Pi_h$ associated to the stabilization form $s_{h,m}$ defined by (\ref{eq:sh-functional}) with 
$m=1$ and $\alpha_1 = d-2$ satisfies 
\eqref{eq:disc_strong_cons} and \eqref{eq:LF-approx} as was shown in \eqref{eq:strong-consistency} and \eqref{eq:Pih-est}.
\end{rem}
%\begin{prop} If the stabilization form $s_h$ is defined by (\ref{eq:sh-functional}) with $\alpha = d-2$.
% Then for $\tau \gtrsim 1$,
%\begin{equation}
%\tn u - u_h \tn_h \lesssim h^k \| u \|_{H^{k+1}(\Omega)}
%\end{equation}
%\end{prop}

%\begin{proof} Splitting the error by adding and subtracting the interpolant $\Pi_h u$, 
%with $\Pi_h$ defined in (\ref{eq:Pih}),  we have  using the triangle inequality 
%\begin{align}
%\tn u - u_h \tn_h \leq \tn u - \Pi_h u \tn_h + \tn \Pi_h u - u_h \tn_h 
%\lesssim  h^k \| u \|_{H^{k+1}(\Omega)} +  \tn \Pi_hu - u_h \tn_h 
%\end{align}
%For the second term we use the coercivity as follows
%\begin{align}\label{eq:infsup-b}
%\tn \Pi_hu - u_h \tn_h \lesssim \sup_{v \in V_h }\frac{A_h(\Pi_hu - u_h,v)}{\tn v \tn_h}
%\end{align}
%where we,  thanks to (\ref{eq:strong-consistency}), have the identity 
%\begin{align}
%A_h(\Pi_hu - u_h,v) &= A_h(\Pi_hu, v ) - l_h(v) = a_h(\Pi_h u, v) + s_h(\Pi_h u , v ) - l_h(v)
%\\
%&\qquad =  a_h(\Pi_hu, v ) -  a_h(u,v) = a_h( \Pi_hu - u, v) 
%\end{align}
%Estimating the right hand side
%\begin{align}
%|a_h(\Pi_hu - u_h,v)| \lesssim \tn  \Pi_hu - u \tn_h \tn v \tn_h 
%\lesssim h^{k} \| u \|_{H^{k+1}(\Omega)} \tn v \tn_h
%\end{align}
%the desired estimate follows from (\ref{eq:infsup-b}).
%\end{proof}

\begin{prop} If the stabilization form $s_{h,\tau}$ is defined by (\ref{eq:sh-functional}) with $\alpha = d-2$.  Then 
\begin{equation}
\lim_{\tau \rightarrow \infty} u_h = \widetilde{\Pi}_h u_h  
\end{equation}
\end{prop}
\begin{proof} First we observe that it follows from the definition (\ref{eq:tildePih-def}) of $\widetilde{\Pi}_h$ that 
$\varphi_i^*(\widetilde{\Pi}_h v) = \varphi_i^*(v)$ for all $i \in I^L$.  Thus the nodal values corresponding to indices $i \in I^L$ do not change and it remains to study the nodal values corresponding to $i \in I^S$.  Setting $v = u_h$ in (\ref{eq:fem}) and using coercivity we get 
\begin{equation}
l_h(u_h) = A_h(u_h,u_h) \gtrsim \tn u_h \tn_{h,\bigstar}^2 = \tn u_h \tn_h^2 + \| u_h \|^2_{s_{h,\tau}}
\end{equation}
Furthermore, we have 
\begin{align}
|l_h(v)| &\lesssim \| f \|_\Omega \| v \|_\Omega  \lesssim  \| f \|_\Omega \tn v \tn_h
\end{align}
where we used the Poincar\'e estimate $\| v \|_\Omega \lesssim \tn v \tn_h$ for $v \in V_h$.  Combining the estimates we get 
\begin{align}
\tn u_h \tn_h^2 + \| u_h \|^2_{s_{h,\tau - \tau_0}} \leq C \| f \|_\Omega \tn v \tn_h 
\leq \frac12 C^2 \| f \|^2_\Omega  +  \frac12 \tn v \tn_h^2
\end{align}
which directly implies
\begin{align}
\frac12 \tn u_h \tn^2 + \| u_h \|^2_{s_{h,\tau}} \leq \frac12 C^2 \| f \|_\Omega^2 
\end{align}
Next using the special choice  (\ref{eq:sh-functional}) of the stabilization form we have 
\begin{align}
\|v \|^2_{s_{h,\tau}}
= \sum_{i\in I^S}   \tau h^{d-2} |\varphi^*_{i, T_i} ([v]_{T_i,S_h(T_i)})|^2
\end{align}
for $v \in V_h$, and from definition (\ref{eq:tildePih-def}) of $\widetilde{\Pi}_h$ we obtain 
the identity 
\begin{equation}
 \varphi^*_{i, T_i}([v]_{T_i,S_h(T_i)}) =  \varphi^*_{i, T_i}(v - (v|_{S_h(T_i)})^e) = \varphi^*_{i, T_i}(v - \widetilde{\Pi}_h v)
\end{equation}
%and therefore 
%\begin{align}
%s_{h,\tau - \tau_0}(v,w) 
%= \sum_{i\in I^S}   (\tau - \tau_0) h^{d-2} \varphi^*_{i, T_i} (v - \Pi_h v) \varphi^*_{i,T_i}(w - \Pi_h w)
%\end{align}
which gives
\begin{align}
\| v \|^2_{s_{h, \tau}}  = \sum_{i\in I^S} \tau h^{d-2} | \varphi_{i,T_i}^*( v - \widetilde{\Pi}_h v)) |^2 
\end{align}
We conclude that 
\begin{align}
 \sum_{i\in I^S} h^{d-2} | \varphi_{i,T_i}^*( u_h - \widetilde{\Pi}_h u_h ) |^2  \lesssim \frac{1}{\tau} \|f \|_\Omega^2
\end{align}
and the desired result follows.
\end{proof}

\section{Numerical Examples}

In the numerical examples below, we compare the forms taken from Example 1,
(face penalty, (\ref{eq:face-penalty})), from 
Example 3 ($L^2$ penalty on the gradient of the difference of the solution and local extension from the interior, (\ref{eq:gradstab}))
and from the nodal stabilization (\ref{eq:sh-functional}). We show that face penalty is most sensitive to locking for large $\tau$ and that the nodal stabilization is completely robust.

We consider the Poisson's equation on a circle of radius $1/2$ with center at the origin. On this circle we use the constructed solution
\begin{equation}
u=\cos(\pi r), \quad r=\sqrt{x^2+y^2}
\end{equation}
corresponding to the right--hand side
\begin{equation}
f=(\pi(\sin(\pi r) + \pi r\cos(\pi r)))/r
\end{equation}
We set the penalty parameter in the different cases to $\tau\in \{10^{-1},10,10^3\}$ and display convergence and condition numbers for the different approaches for increasing $\tau$.

In Figs. \ref{point:conv1}--\ref{point:conv2} we give convergence plots for the point oriented method. We note that the convergence
pattern is unaffected by the increased penalty; in Figs. \ref{l2:conv1}--\ref{l2:conv2} we give the corresponding results for the gradient penalty and in Figs. \ref{edge:conv1}--\ref{edge:conv2} for the face penalty. The dashed and dotted lines, indicating first and second-order convergence, respectively, are fixed in all diagrams. We note the slight locking for the gradient penalty and severe locking for edge stabilization on coarse meshes. This is also visible in the elevation plots on coarse and fine meshes with $\tau=10^3$, for point oriented in Fig.
\ref{point:elev}, for gradient penalty in Fig. \ref{l2:elev}, where small instabilities near the boundary are visible on the coarse mesh, and
for the face penalty in Fig. \ref{edge:elev}, which exhibits severe locking on coarse meshes and has a visible effect also on the fine mesh.

Finally, in Figs. \ref{point:cond}, \ref{l2:cond}, and \ref{edge:cond}, we give the plots of the condition number for the different methods, for $\tau=10^{-1}$ and for $\tau=10^3$. The rate given by the dashed line is $O(h^{-2})$, and we note that all methods give approximately the same conditioning of the discrete system, with increasing $\tau$ adversely affecting the condition number. \textcolor{black}{In all graphics, the natural logarithm of the plotted quantity is reported on the ordinate.}

\section{Appendix}
\appendix
\section{$\boldmath{L}^{\boldmath{2}}$ Stability for Nodal Stabilization}
Here we include the simplified proof of the stability estimate (\ref{eq:nodalstability}) in the case $m=0$,
\begin{align}
 \| v \|^2_{\Omega_h} \lesssim \| \nabla^m v \|^2_\Omega + \| v \|^2_{s_{h,0}}, \quad v \in V_h
\end{align}
for the nodal stabilization $s_{h,0}$ defined by  (\ref{eq:sh-functional}) with $m=0$.

\begin{proof}
We have 
\begin{align}
\|  v \|^2_{\Omega_h} &= \| v \|^2_{\mcT_h^S} + \| v \|^2_{\mcT_h^L}
\\
&\lesssim  h^d \sum_{T \in \mcT_h^S}  \| \hatv_T \|^2_{\IR^{N_T}} + \| v \|^2_{\mcT_h^L}
\\
&\lesssim  h^d\| \hatv^S \|^2_{\IR^{N_S}} + \| v \|^2_{\mcT_h^L}
\end{align}
Thus it follows that,  if the stabilization form satisfies 
\begin{align}\label{eq:stab-property-a}
h^d \| \hatv^S \|^2_{\IR^{N_S}}  \lesssim \| v \|^2_{\mcT_h^L} + \| v \|^2_{s_{h,0}} 
\end{align}
the desired stability estimate (\ref{eq:stab-est}) holds.  To verify (\ref{eq:stab-property-a}) we start from the 
definition (\ref{eq:sh-functional}) of the stabilization form.  For each $i \in I_S$,  we have the identity 
\begin{align}
\hatv_i = \varphi_{i,T_i}^*(v) =  \varphi_{i,T_i}^*((v|_{S_h(T_i)})^e) +  \varphi_{i,T_i}^*([v]_{T_i,S_h(T_i)})
\end{align}
Using the triangle inequality we get 
\begin{align}
h^{d} |\hatv_i|^2 &=h^{d} | \varphi_{i,T_i}^*(v)|^2  
\\
&\lesssim   h^d | \varphi_{i,T_i}^*((v|_{S_h(T_i)})^e) |^2 
+ h^d | \varphi_{i,T_i}^*([v]_{T_i,S_h(T_i)}) |^2
\\
&\lesssim   \| (v|_{S_h(T_i)})^e \|^2_T 
+ h^d | \varphi_{i,T_i}^*([v]_{T_i,S_h(T_i)}) |^2
\\
&\lesssim  \| v \|^2_{S_h(T_i)}
+ h^d | \varphi_{i,T_i}^*([v]_{T_i,S_h(T_i)}) |^2
\end{align}
where we used an inverse inequality and the stability of the canonical extension.  Summing over all $i \in I^S$ 
we get 
\begin{align}
h^d \| \hatv^S \|^2_{\IR^{N_S}}  &= \sum_{i \in I^S} h^{d} |\hatv_i|^2 
\\
&\lesssim \sum_{i \in I^S}  \| v \|^2_{S_h(T_i)}
+ h^d | \varphi_{i,T_i}^*([v]_{T_i,S_h(T_i)}) |^2
\\
&\lesssim \| v \|^2_{\mcT_h^L}
+  \sum_{i \in I^S}  h^d | \varphi_{i,T_i}^*([v]_{T_i,S_h(T_i)}) |^2
\\
&\lesssim \| v \|^2_{\Omega}
+ \| v \|^2_{s_{h,0}}
\end{align}
as was to be shown.
\end{proof}

\section{Coercivity}
\label{app:coer}

\begin{lem}\label{lem:nitsche-coer}
The form $A_h$, defined in (\ref{eq:Ah}) is coercive on $V_h$ for sufficiently large 
$\beta\in [\beta_0, \infty)$,
\begin{align}\label{eq:nitsche-coer-lem}
\tn v \tn_{h,\bigstar}^2 \lesssim A_h(v,v) \qquad v \in V_h
\end{align}
and $\tau \in [\tau_0,\infty)$ where $\tau_0 \gtrsim \beta^{-1}$. 
\end{lem}
\begin{proof}[Proof of (\ref{eq:nitsche-coer})] We have
\begin{align}
A_h(v,v) &= \|\nabla v \|^2_\Omega + \| v \|^2_{s_{h,\tau}} - 2(\nabla_n v,v )_{\partial \Omega} + \beta h^{-1} \|v \|^2_{\partial \Omega}
\\
&\geq \|\nabla v \|^2_\Omega + \| v \|^2_{s_{h,\tau}} - 2|(\nabla_n v,v )_{\partial \Omega}| + \beta h^{-1} \|v \|^2_{\partial \Omega}
\end{align}
Recalling (\ref{eq:nitsche-inverse}), we have 
\begin{equation}
h \|\nabla_n v \|^2_{\partial \Omega} %\lesssim \| \nabla v \|^2_{\mcT_h(\partial \Omega)}
\lesssim  \| \nabla v \|^2_\Omega + \|v \|^2_{s_{h,\tau_*}}
\end{equation}
for some fixed positive parameter $\tau_*$, and we may estimate the negative term as follows
\begin{align}
2|(\nabla_n v,v )_{\partial \Omega}| &\leq  2 \beta^{-1} h \| \nabla_n v \|^2_{\partial \Omega} 
+ 2^{-1} \beta h^{-1} \| v \|^2_{\partial \Omega}
\\
&\leq  2 \beta^{-1} C ( \| \nabla v \|^2_\Omega + \|v \|^2_{s_{h,\tau_*}} )
+ 2^{-1} \beta h^{-1} \| v \|^2_{\partial \Omega}
\end{align}
where $C$ is the hidden constant in (\ref{eq:nitsche-inverse}). Combining the estimates we get
\begin{align}
A_h(v,v) &= 2^{-1} ( \|\nabla v \|^2_\Omega + \| v \|^2_{s_{h,\tau}}  )
+ ( 2^{-1}  -  2 \beta^{-1} C ) \|\nabla v \|^2_\Omega 
\\
&\qquad 
+ ( 2^{-1} \| v \|^2_{s_{h,\tau_0}}  - 2 \beta^{-1} C\| v \|^2_{s_{h,\tau_*}} )
+ 2^{-1} \beta h^{-1} \|v \|^2_{\partial \Omega}
\\ \label{eq:nitsche-coer-a}
&\geq 2^{-1} ( \|\nabla v \|^2_\Omega + \| v \|^2_{s_{h,\tau}} + \beta h^{-1} \|v \|^2_{\partial \Omega} )
\end{align}
if $\beta$ is chosen such that $2^{-1}  -  2 \beta^{-1} C\geq 0$ and $\tau_0$ satisfies
$\tau_0 \geq 4 \beta^{-1} C$. Here we used the simple identity $\alpha_1 \| v \|^2_{s_{h,\tau_1}} + \alpha_2 \| v \|^2_{s_{h,\tau_2}}  = 
\| v \|^2_{s_{h,\alpha_1 \tau_1 + \alpha_2 \tau_2}}$ to conclude that 
\begin{align}
 2^{-1} \| v \|^2_{s_{h,\tau_0}}  - 2 \beta^{-1} C\| v \|^2_{s_{h,\tau_*}} =
  \| v \|^2_{s_{h,2^{-1} \tau_0- 2 \beta^{-1} \tau_*}}
\end{align}
Finally, we use the fact that $\tau\geq \tau_0 \geq c \beta^{-1} \tau_*$ to conclude that 
\begin{align}
&\|\nabla v \|^2_\Omega + \| v \|^2_{s_{h,\tau}} 
\geq 
\|\nabla v \|^2_\Omega + \| v \|^2_{s_{h,c \beta^{-1} \tau_*}}
\geq  
\|\nabla v \|^2_\Omega + c \beta^{-1} \| v \|^2_{s_{h, \tau_*}}
\\
&\qquad 
\geq  \label{eq:nitsche-coer-b}
\min(1, c \beta^{-1}) ( \|\nabla v \|^2_\Omega + \| v \|^2_{s_{h, \tau_*}} )
 \gtrsim 
 \beta^{-1} h \| \nabla_n v \|_{\partial \Omega}^2
 \end{align}
where we used (\ref{eq:nitsche-inverse}) and the estimate $\min(1,c\beta^{-1})
\gtrsim \beta^{-1}$ since $\beta \geq \beta_0>0$. Together the estimates (\ref{eq:nitsche-coer-a}) and (\ref{eq:nitsche-coer-b}) prove the desired result.
\end{proof}
\paragraph{Acknowledgements.}
This research was supported in part by the Swedish Research
Council Grants Nos.\  2017-03911, 2018-05262,  2021-04925,  and the Swedish
Research Programme Essence. EB acknowledges support from EPSRC, grant EP/P01576X/1.

\bibliographystyle{abbrv}
\footnotesize{
\bibliography{ref}
}

\newpage
\begin{figure}
	\begin{center}
	\includegraphics[scale=0.2]{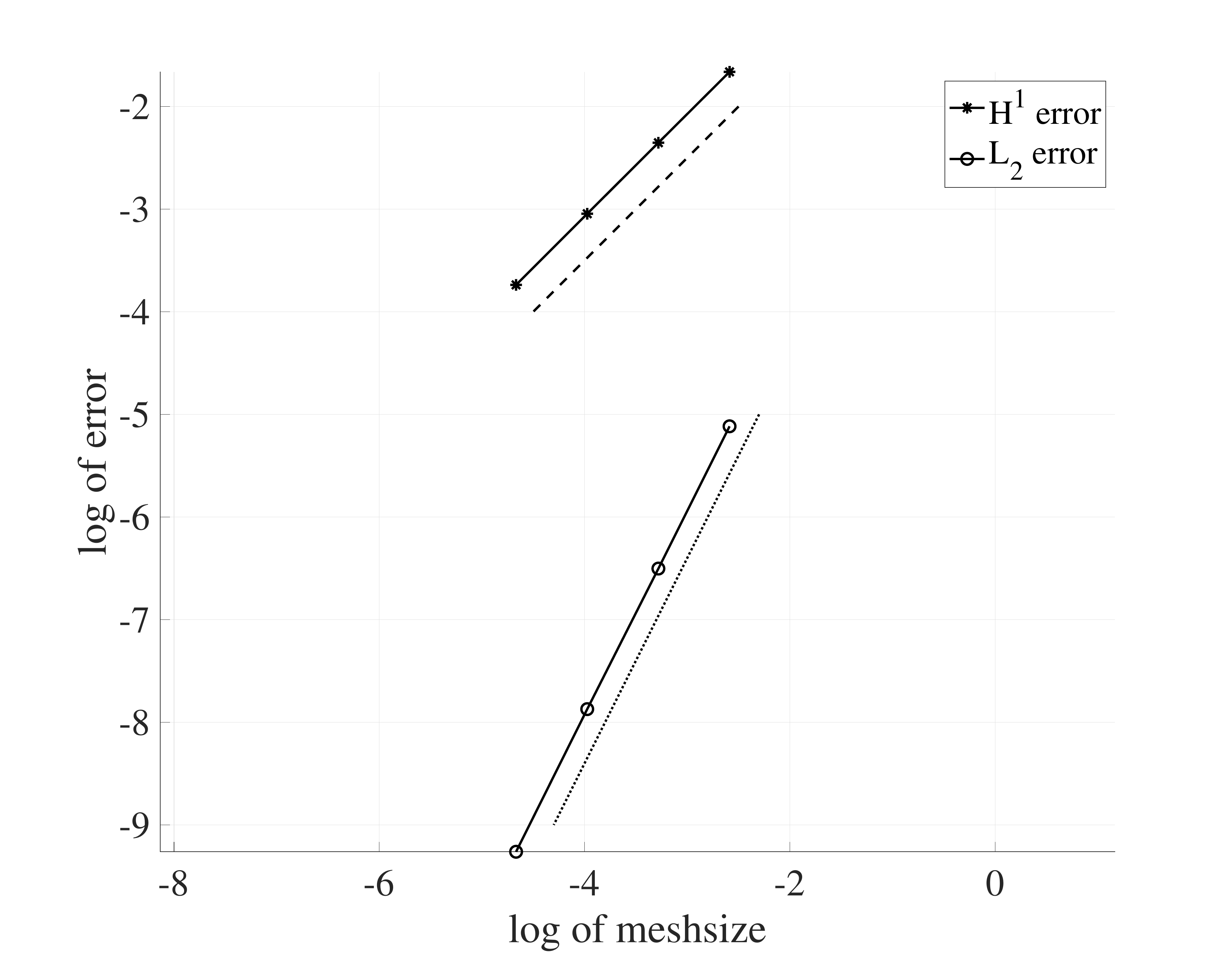}\includegraphics[scale=0.2]{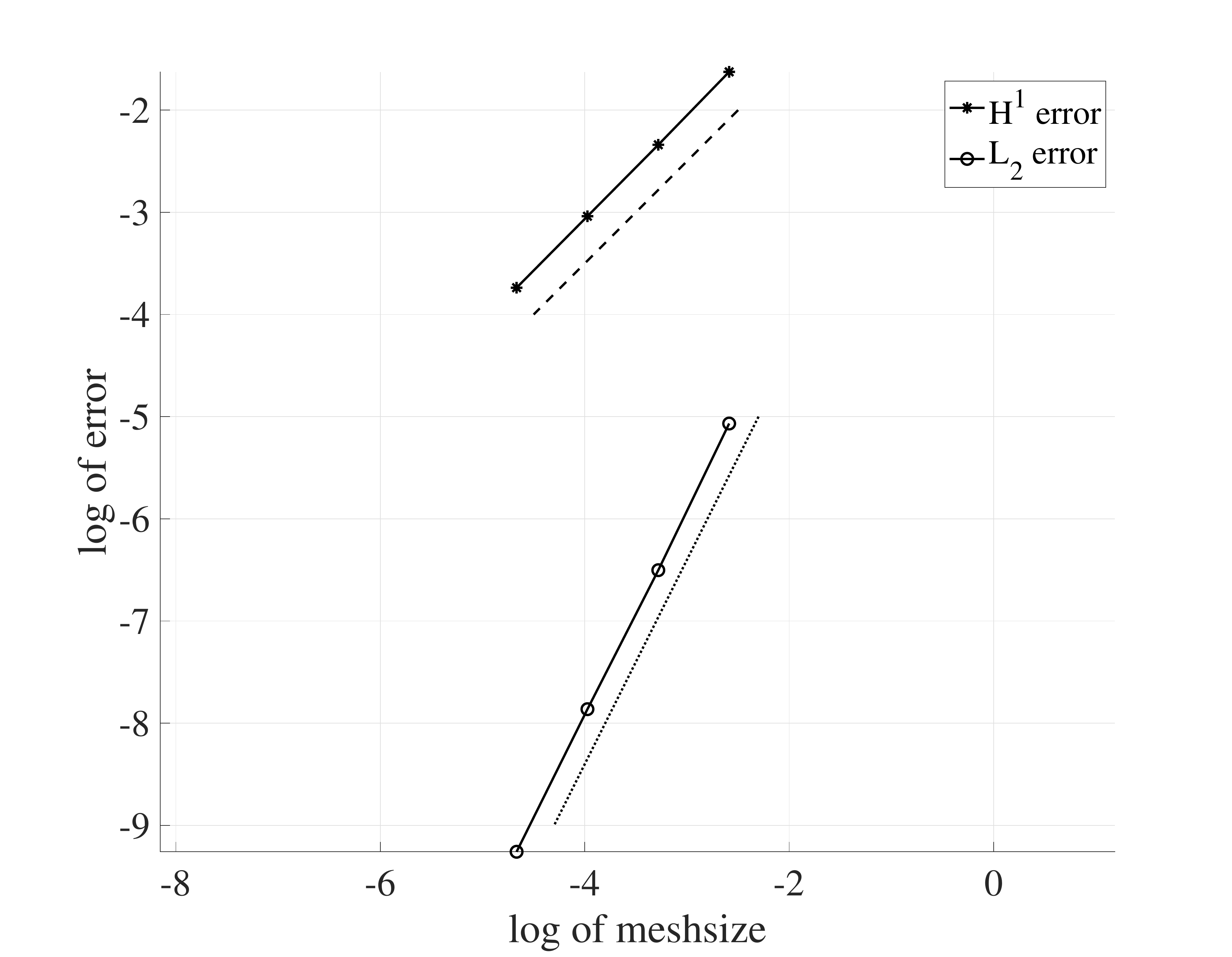}
	\end{center}
\caption{Convergence for point penalty with $\tau = 10^{-1}$ (left) and $\tau=10$ (right).\label{point:conv1}}
\end{figure}
\begin{figure}
	\begin{center}
	\includegraphics[scale=0.2]{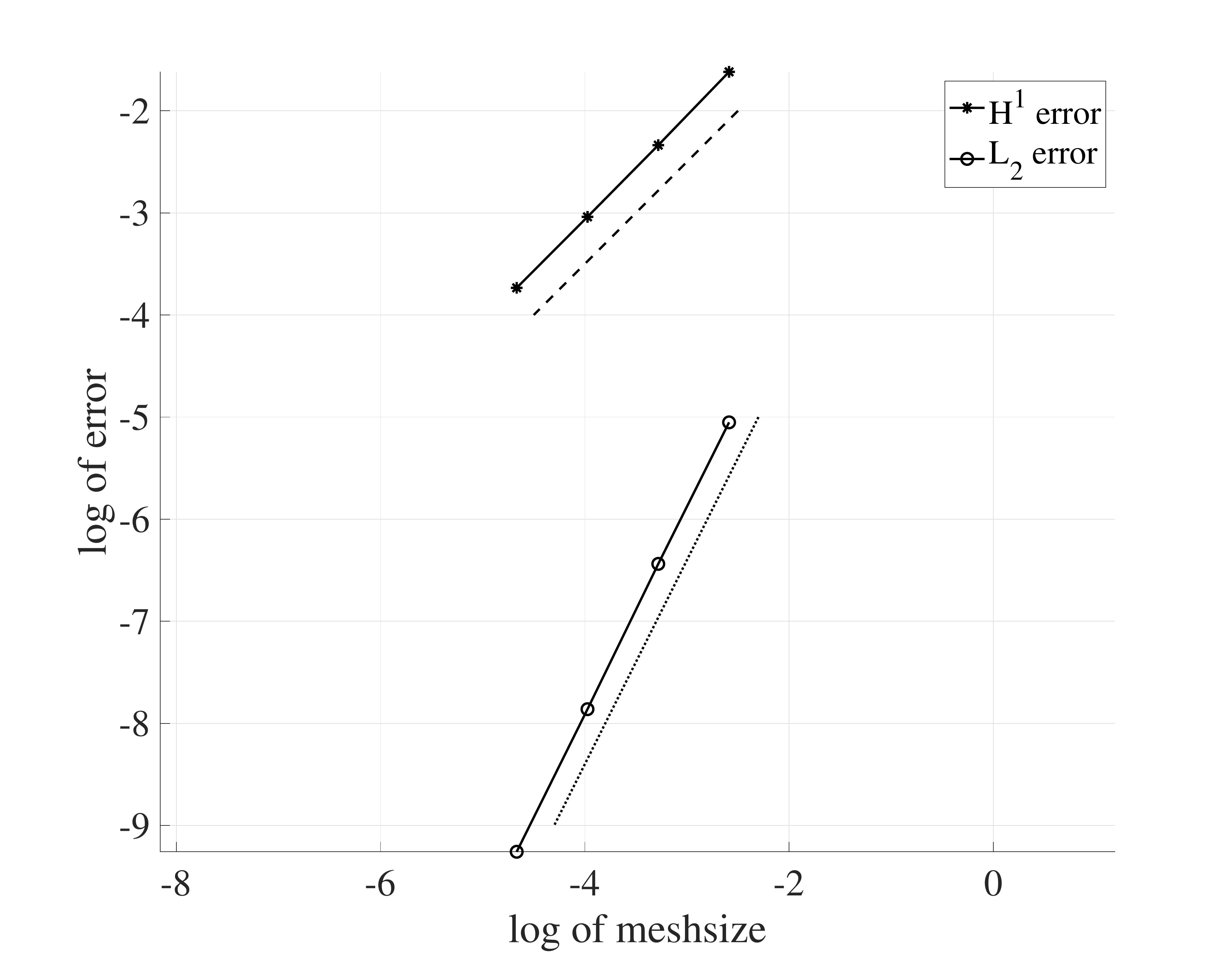}
	\end{center}
\caption{Convergence for point penalty with $\tau = 10^{3}$.\label{point:conv2}}
\end{figure}
\begin{figure}
	\begin{center}
	\includegraphics[scale=0.2]{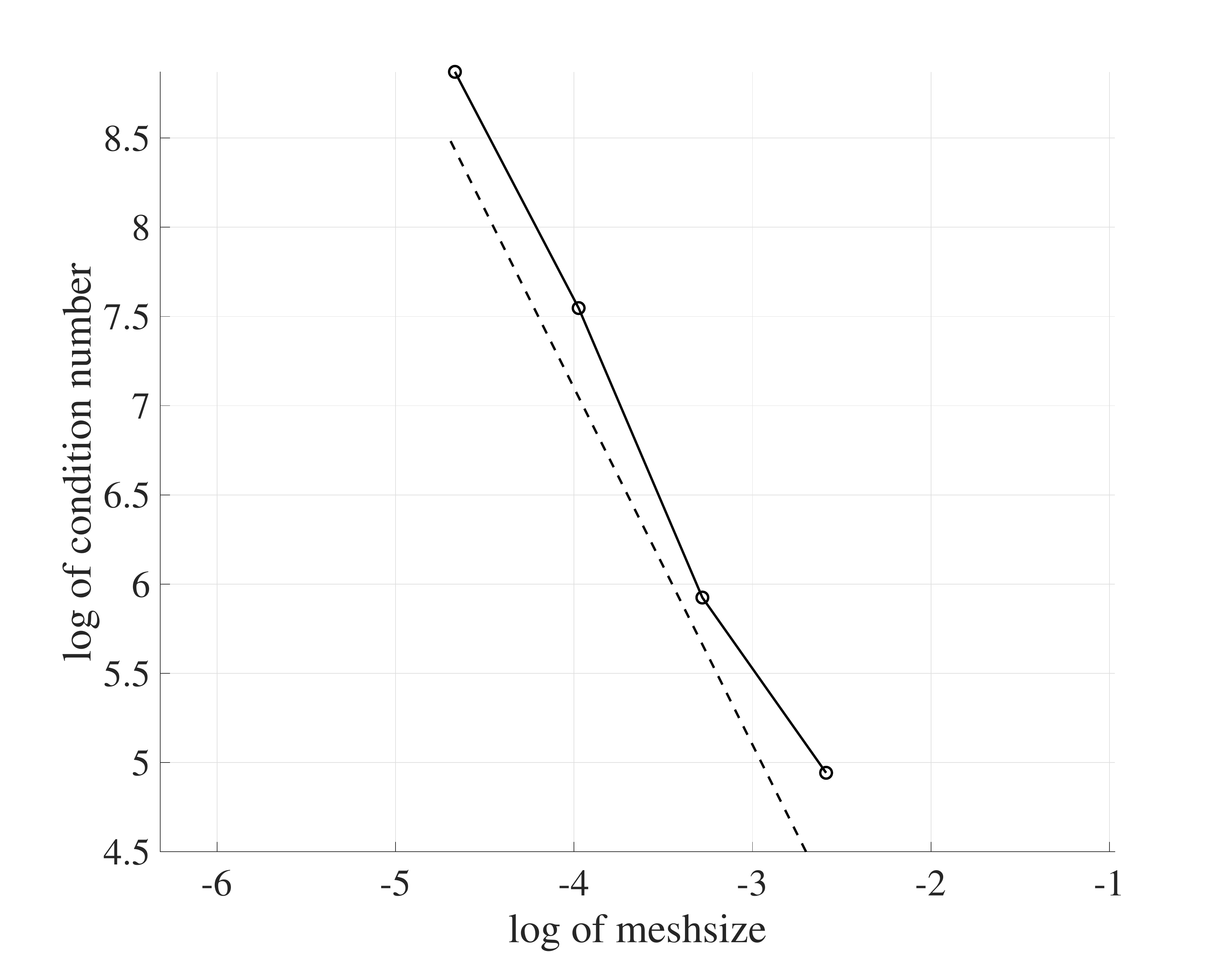}\includegraphics[scale=0.2]{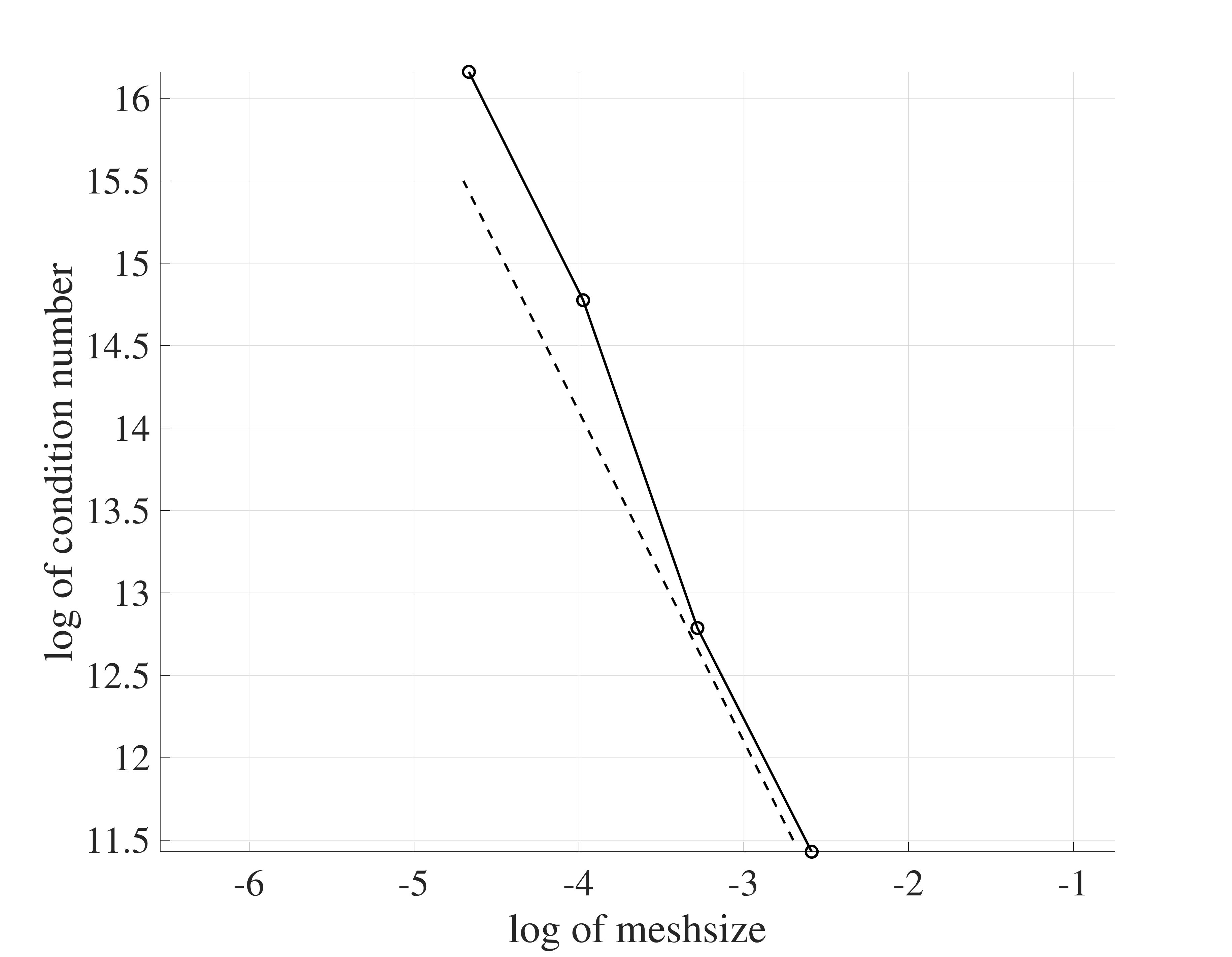}
	\end{center}
\caption{Conditioning for point penalty with $\tau = 10^{-1}$ (left) and $\tau=10^3$ (right).\label{point:cond}}
\end{figure}
\begin{figure}
	\begin{center}
	\includegraphics[scale=0.2]{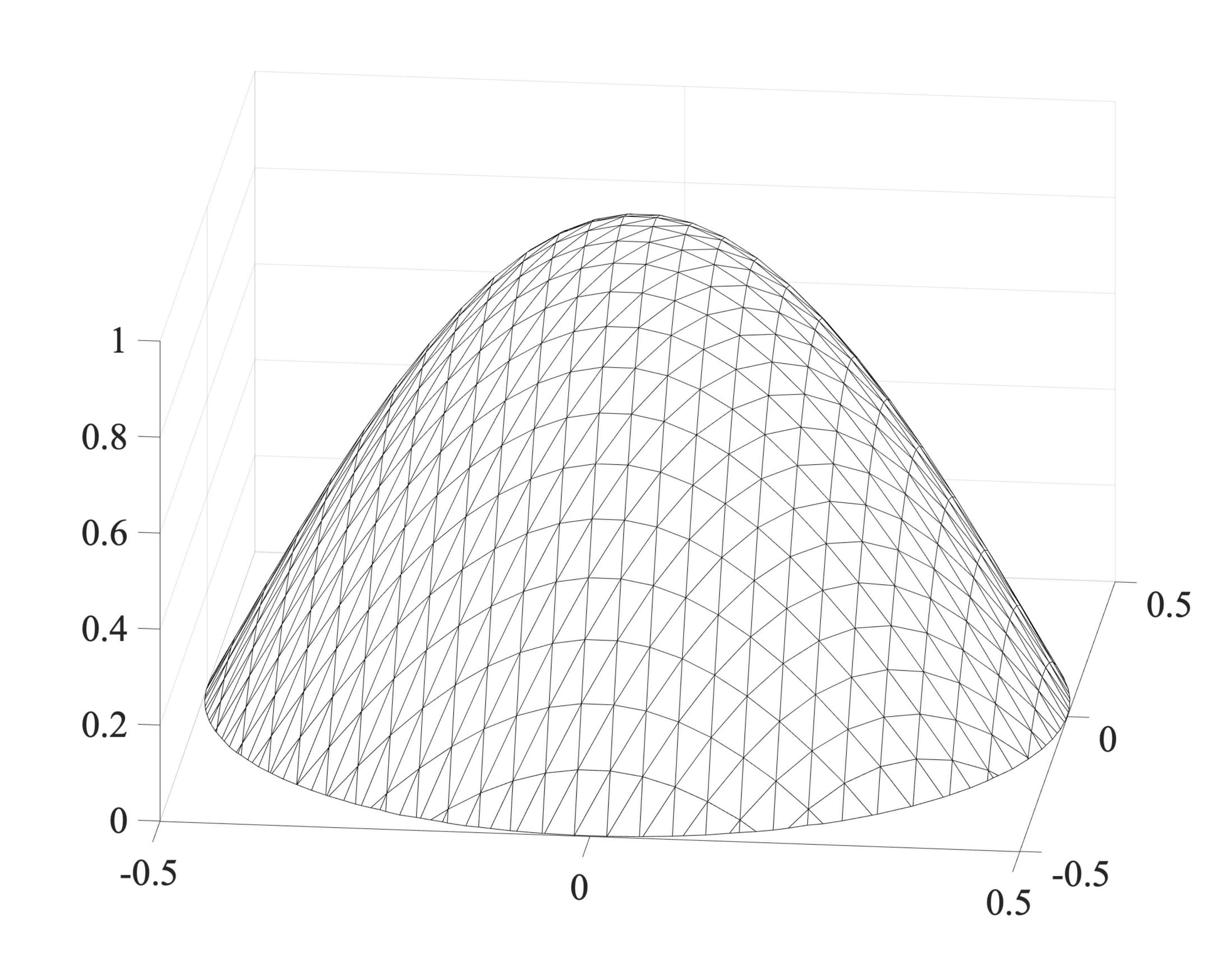}\includegraphics[scale=0.2]{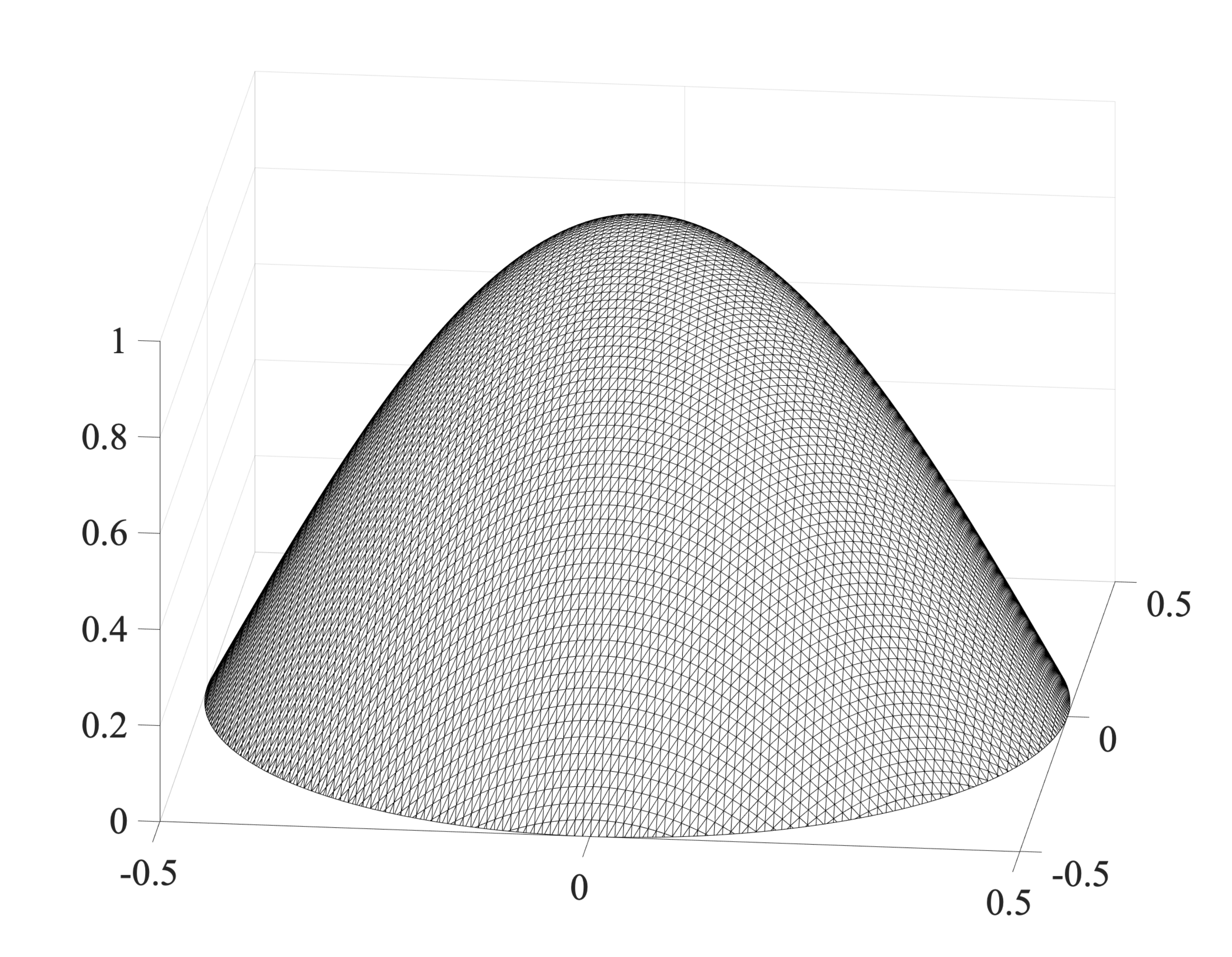}
	\end{center}
\caption{Elevation of the computed solution on a coarse and on a fine mesh for point penalty with $\tau=10^3$.\label{point:elev}}
\end{figure}
\begin{figure}
	\begin{center}
	\includegraphics[scale=0.2]{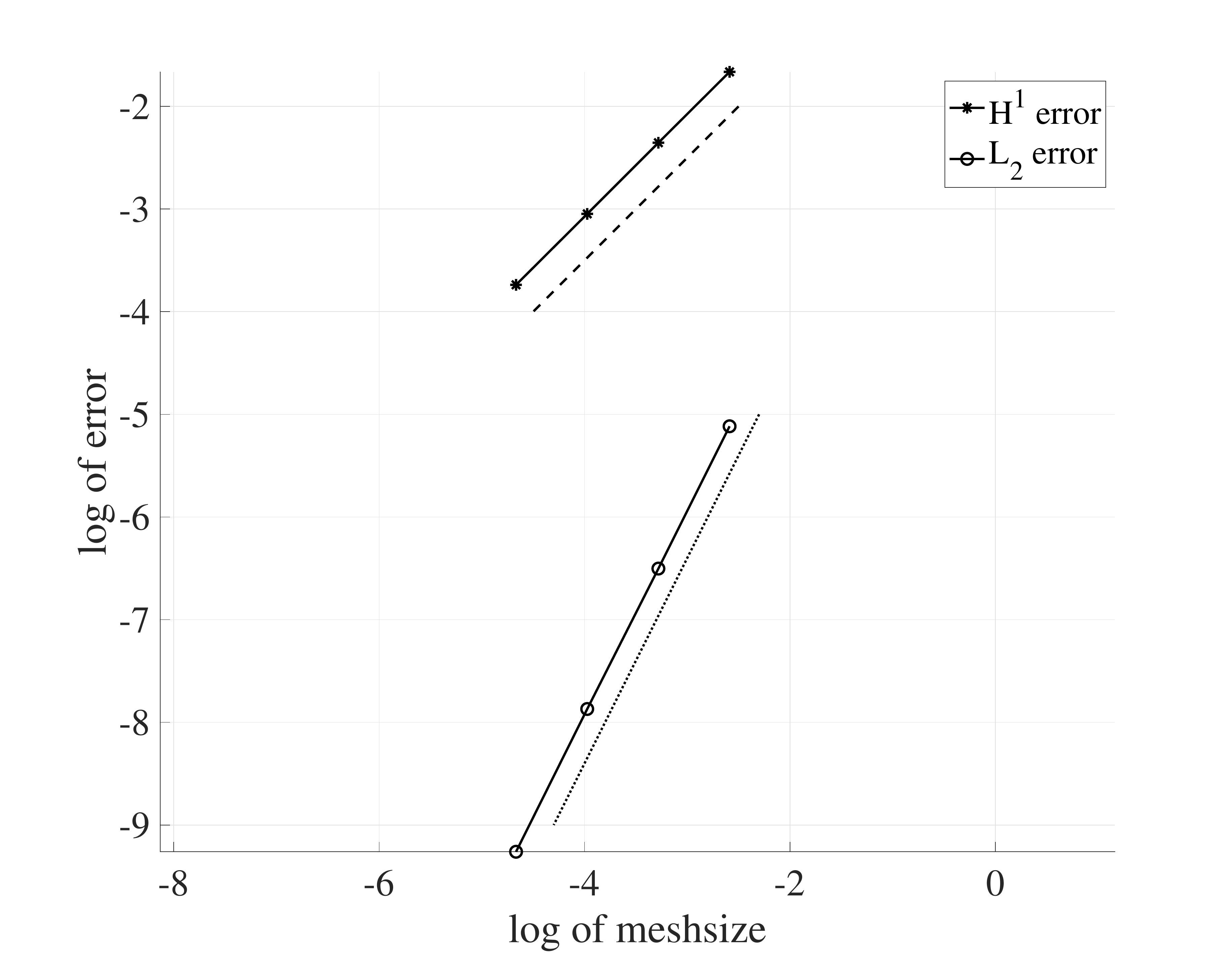}\includegraphics[scale=0.2]{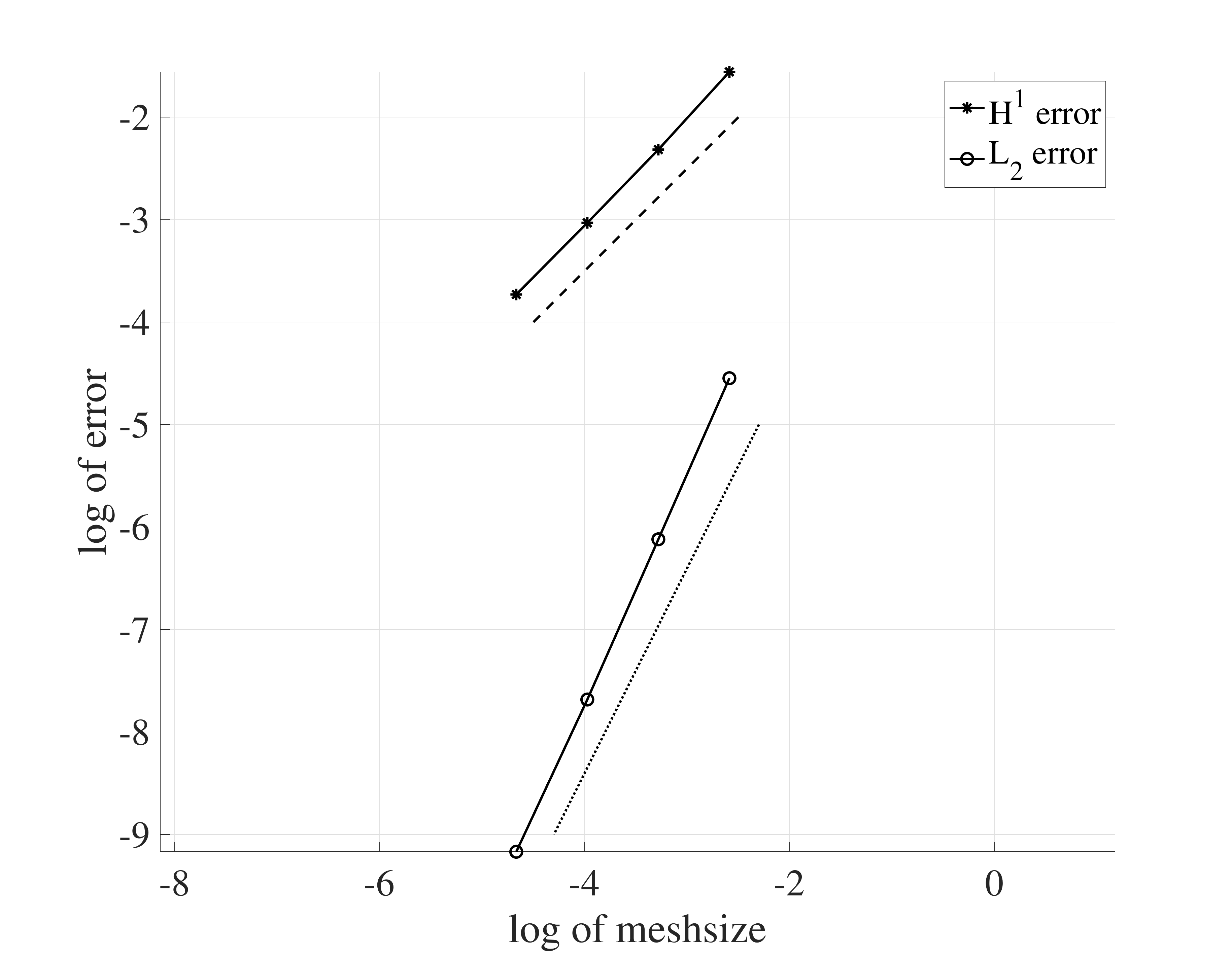}
	\end{center}
\caption{Convergence for $L2$ gradient penalty with $\tau = 10^{-1}$ (left) and $\tau=10$ (right).\label{l2:conv1}}
\end{figure}
\begin{figure}
	\begin{center}
	\includegraphics[scale=0.2]{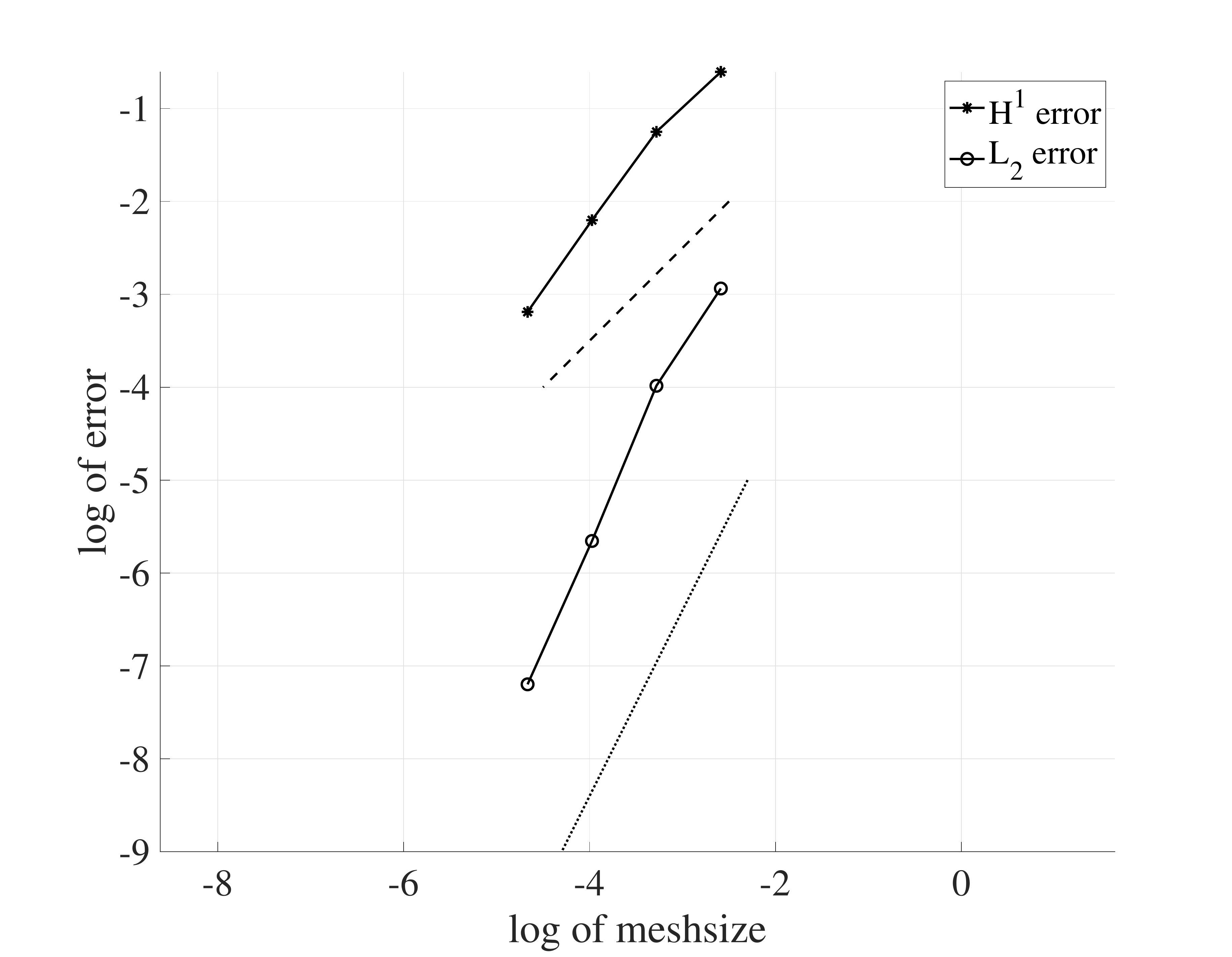}
	\end{center}
\caption{Convergence for $L^2$ gradient penalty with $\tau = 10^{3}$.\label{l2:conv2}}
\end{figure}
\begin{figure}
	\begin{center}
	\includegraphics[scale=0.2]{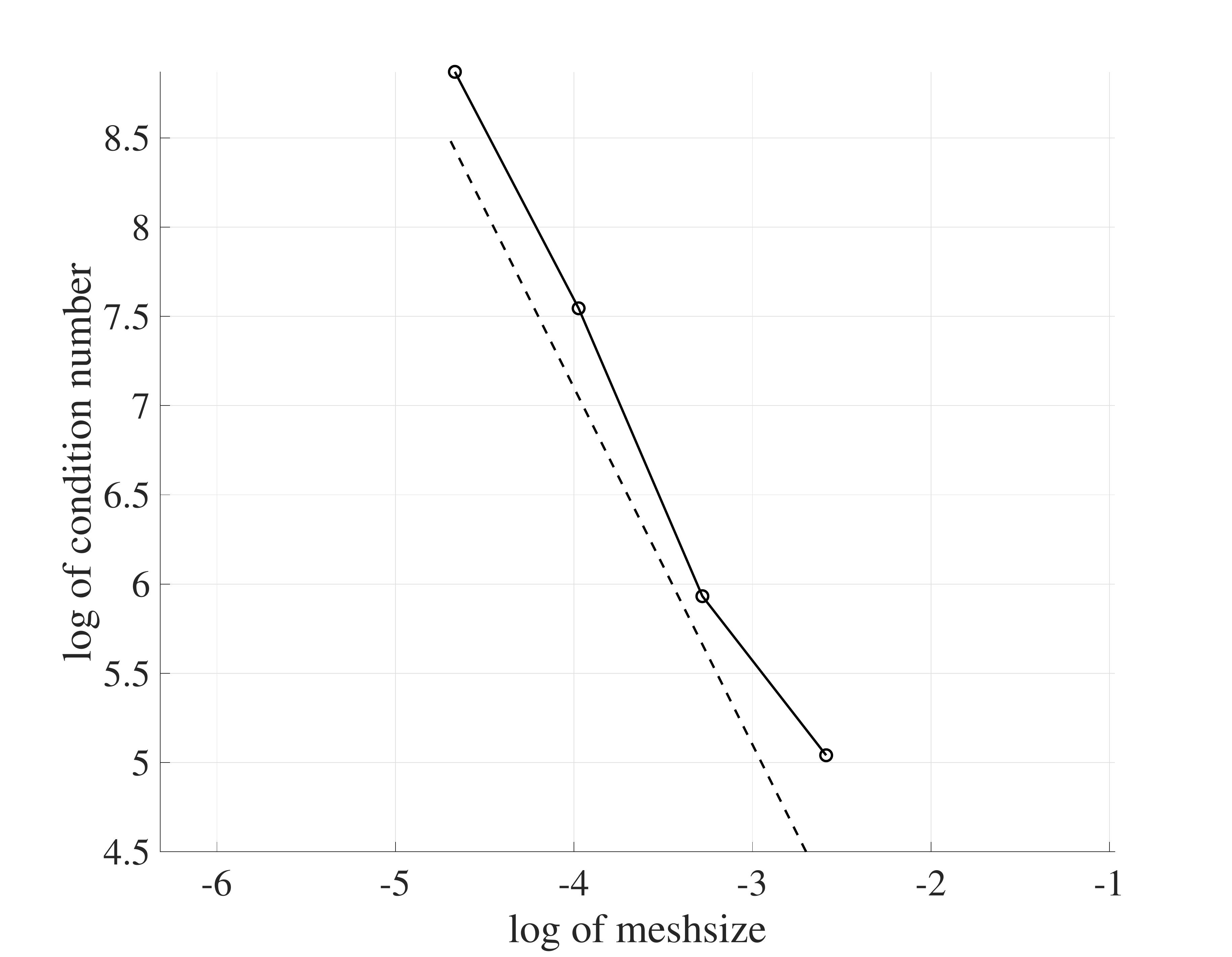}\includegraphics[scale=0.2]{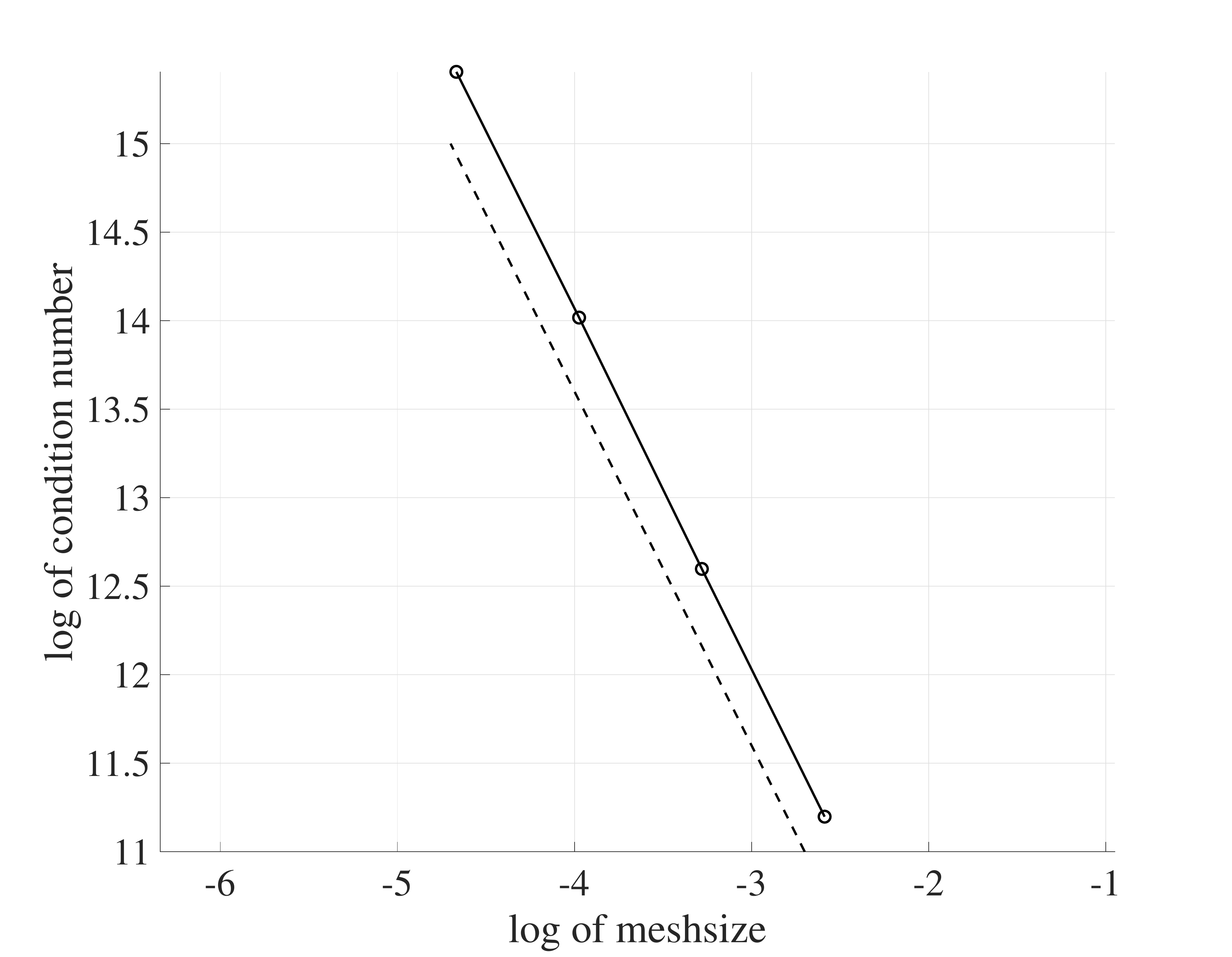}
	\end{center}
\caption{Conditioning for $L^2$ gradient penalty with $\tau = 10^{-1}$ (left) and $\tau=10^3$ (right).\label{l2:cond}}
\end{figure}
\begin{figure}
	\begin{center}
	\includegraphics[scale=0.2]{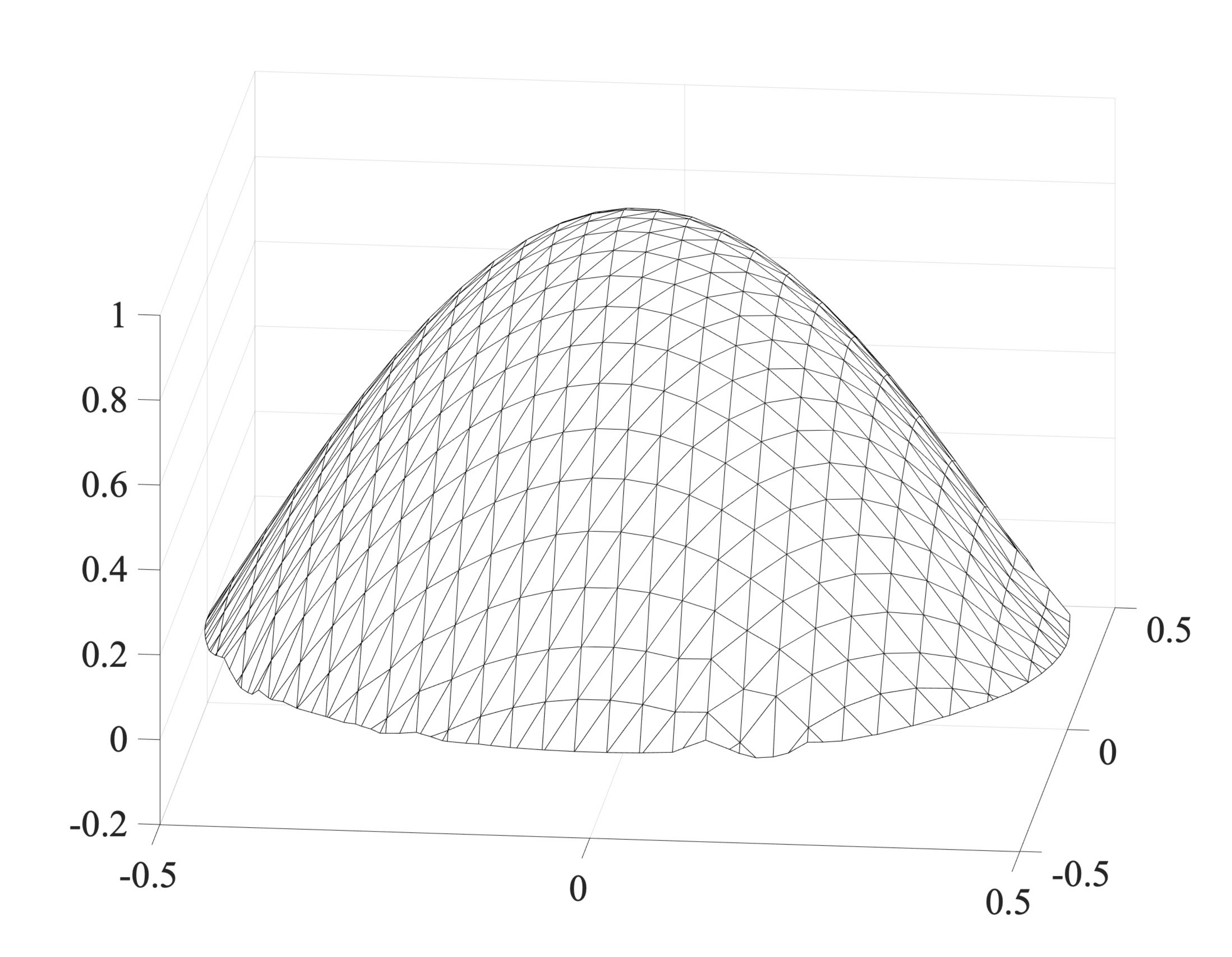}\includegraphics[scale=0.2]{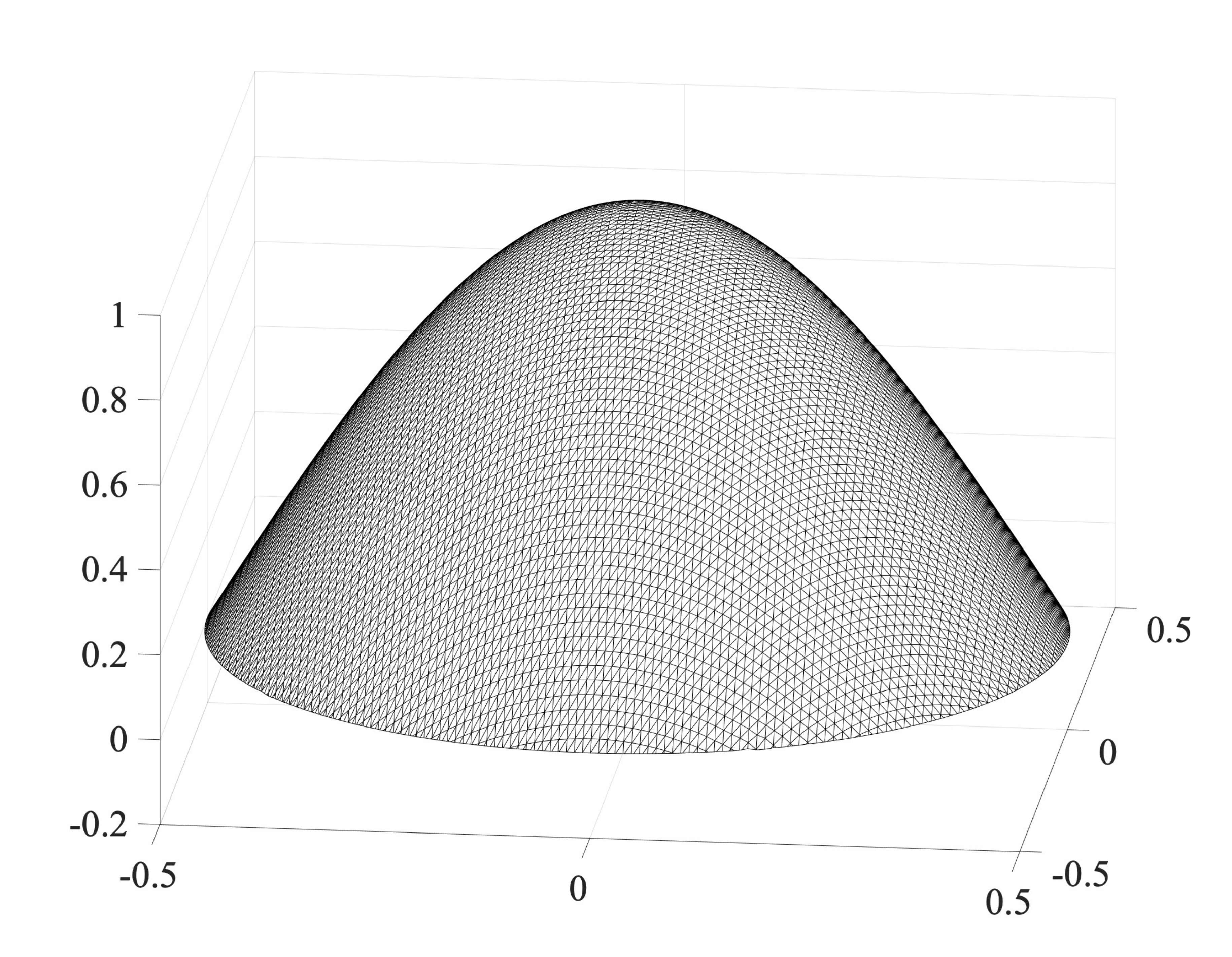}
	\end{center}
\caption{Elevation of the computed solution on a coarse and on a fine mesh for $L^2$ gradient penalty with $\tau=10^3$.\label{l2:elev}}
\end{figure}
\begin{figure}
	\begin{center}
	\includegraphics[scale=0.2]{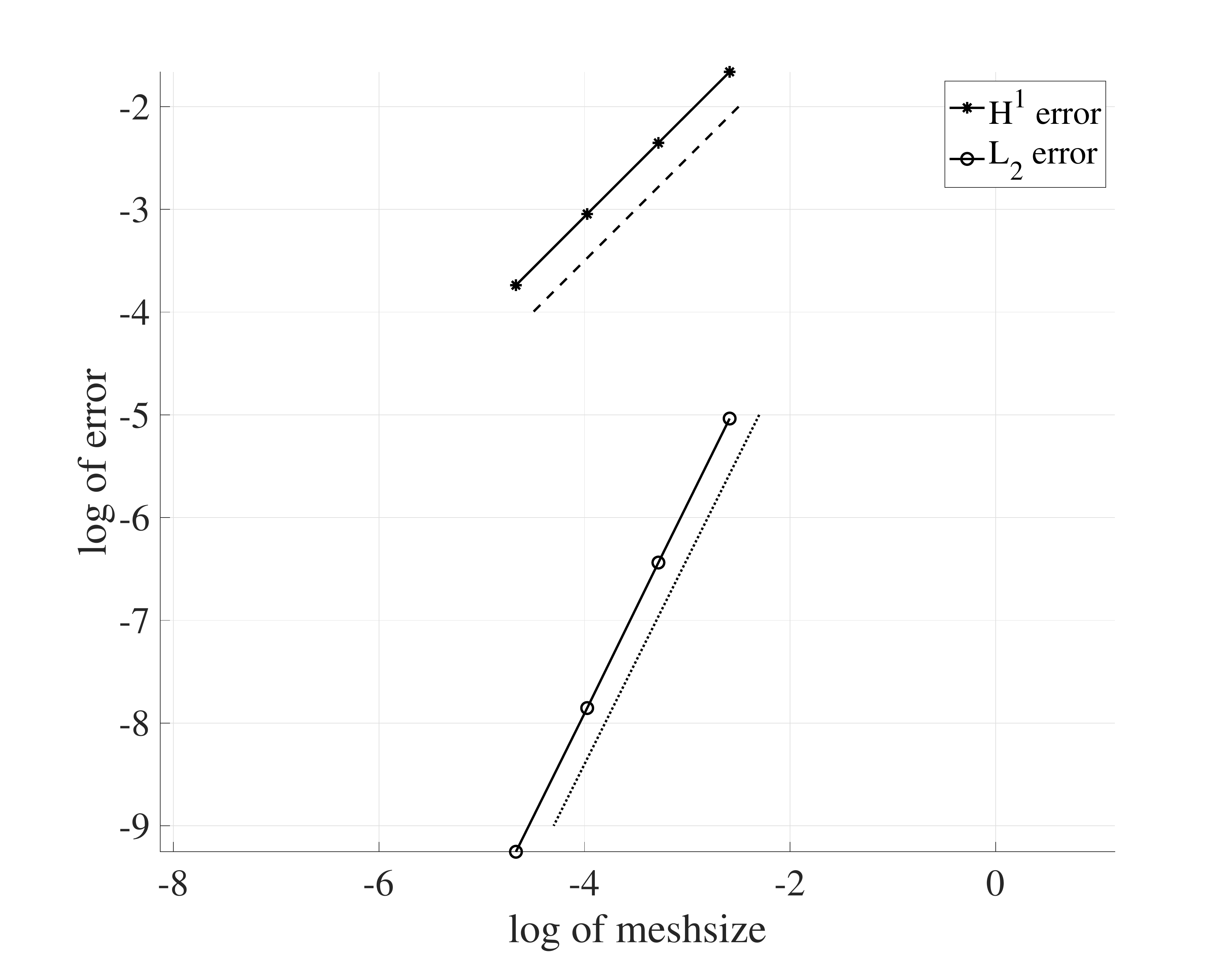}\includegraphics[scale=0.2]{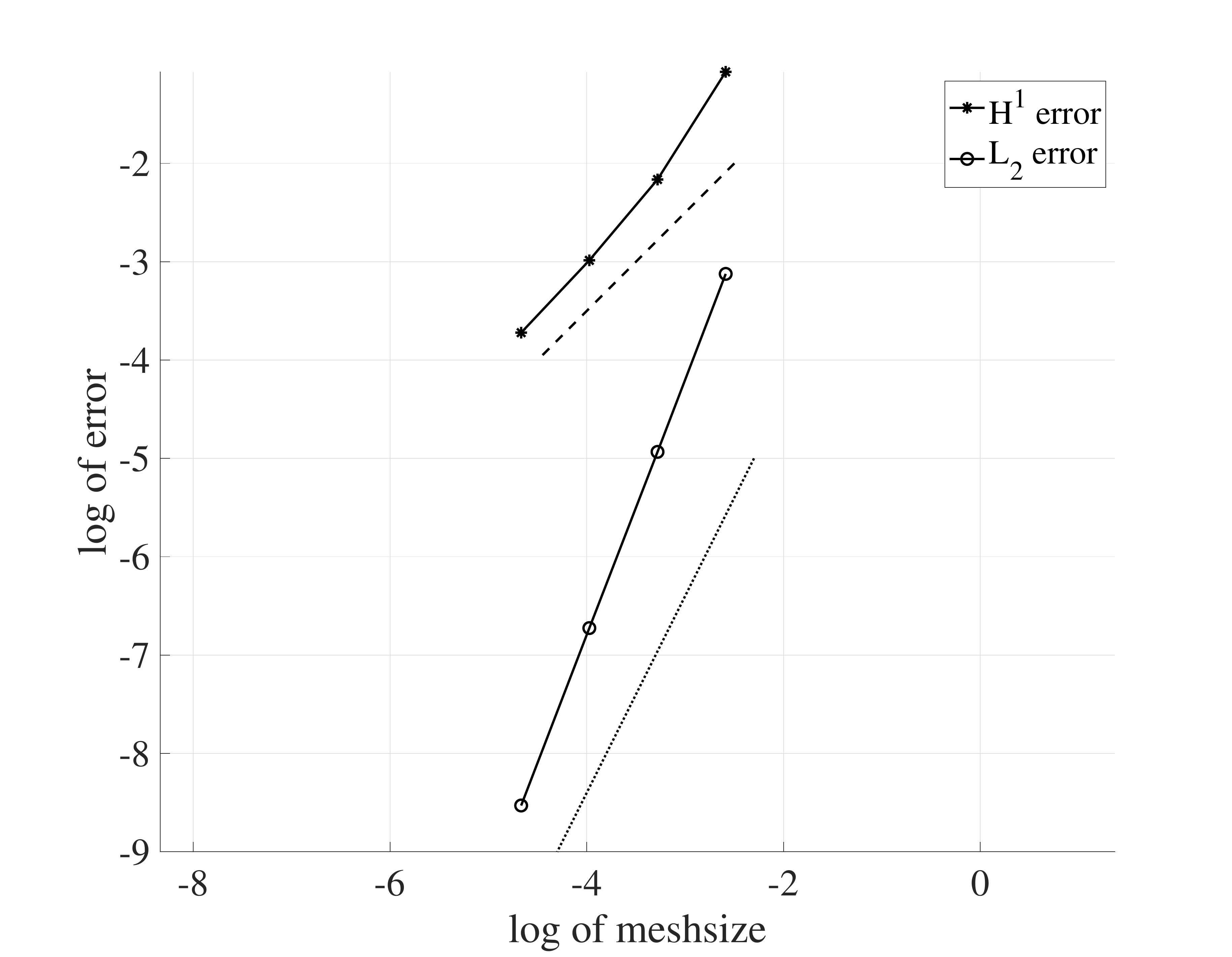}
	\end{center}
\caption{Convergence for face penalty with $\tau = 10^{-1}$ (left) and $\tau=10$ (right).\label{edge:conv1}}
\end{figure}
\begin{figure}
	\begin{center}
	\includegraphics[scale=0.2]{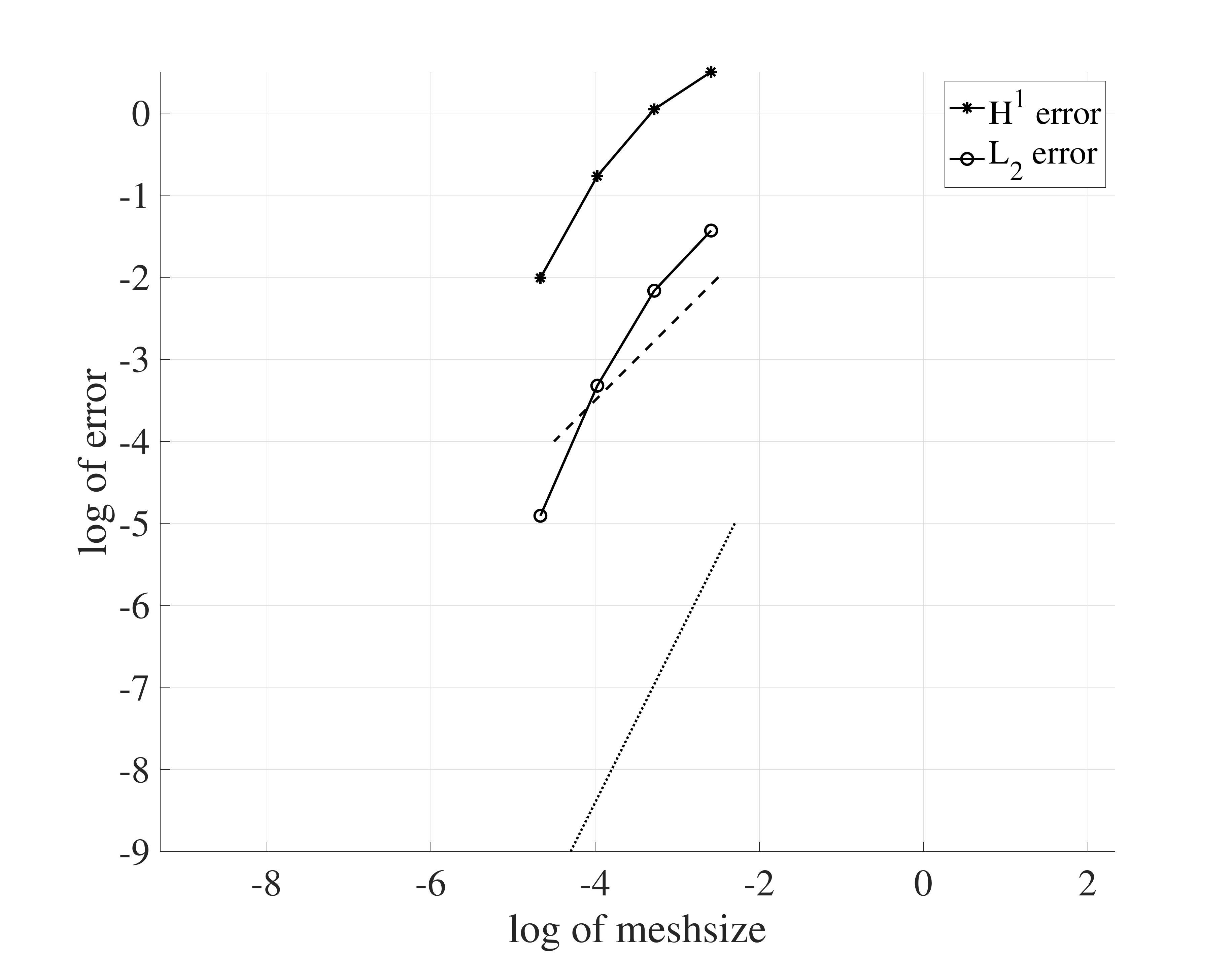}
	\end{center}
\caption{Convergence for face penalty with $\tau = 10^{3}$.\label{edge:conv2}}
\end{figure}
\begin{figure}
	\begin{center}
	\includegraphics[scale=0.2]{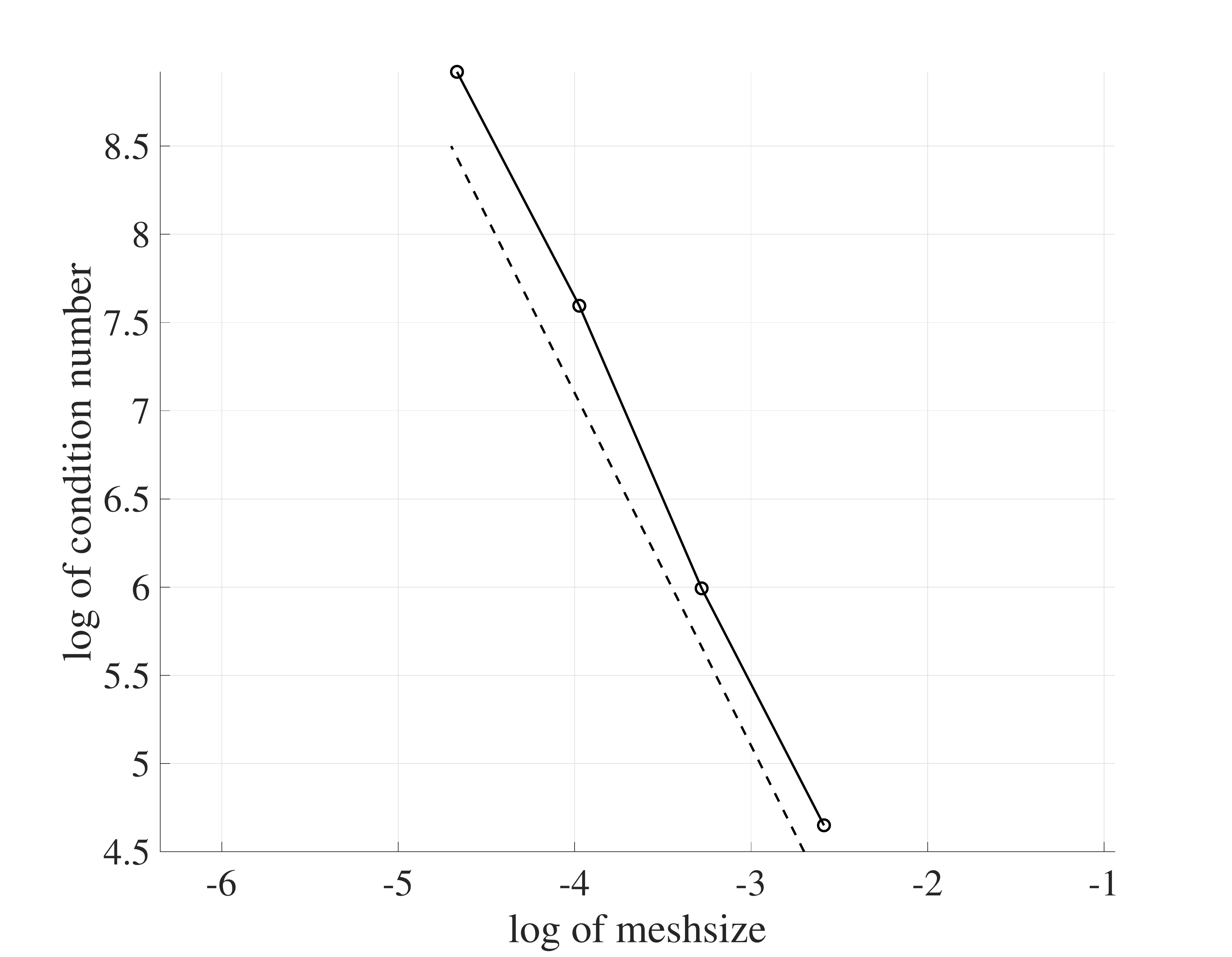}\includegraphics[scale=0.2]{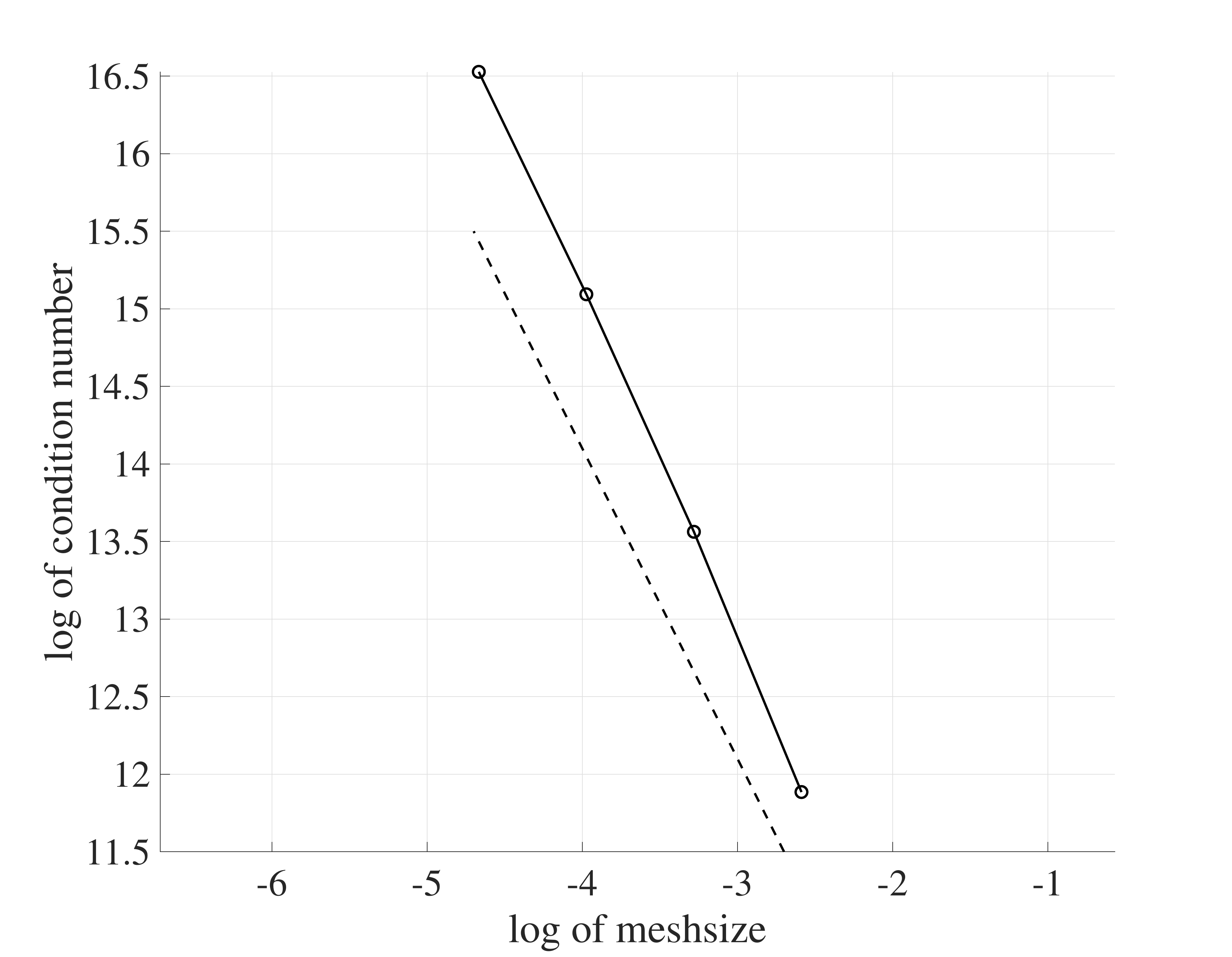}
	\end{center}
\caption{Conditioning for face penalty with $\tau = 10^{-1}$ (left) and $\tau=10^3$ (right).\label{edge:cond}}
\end{figure}
\begin{figure}
	\begin{center}
	\includegraphics[scale=0.2]{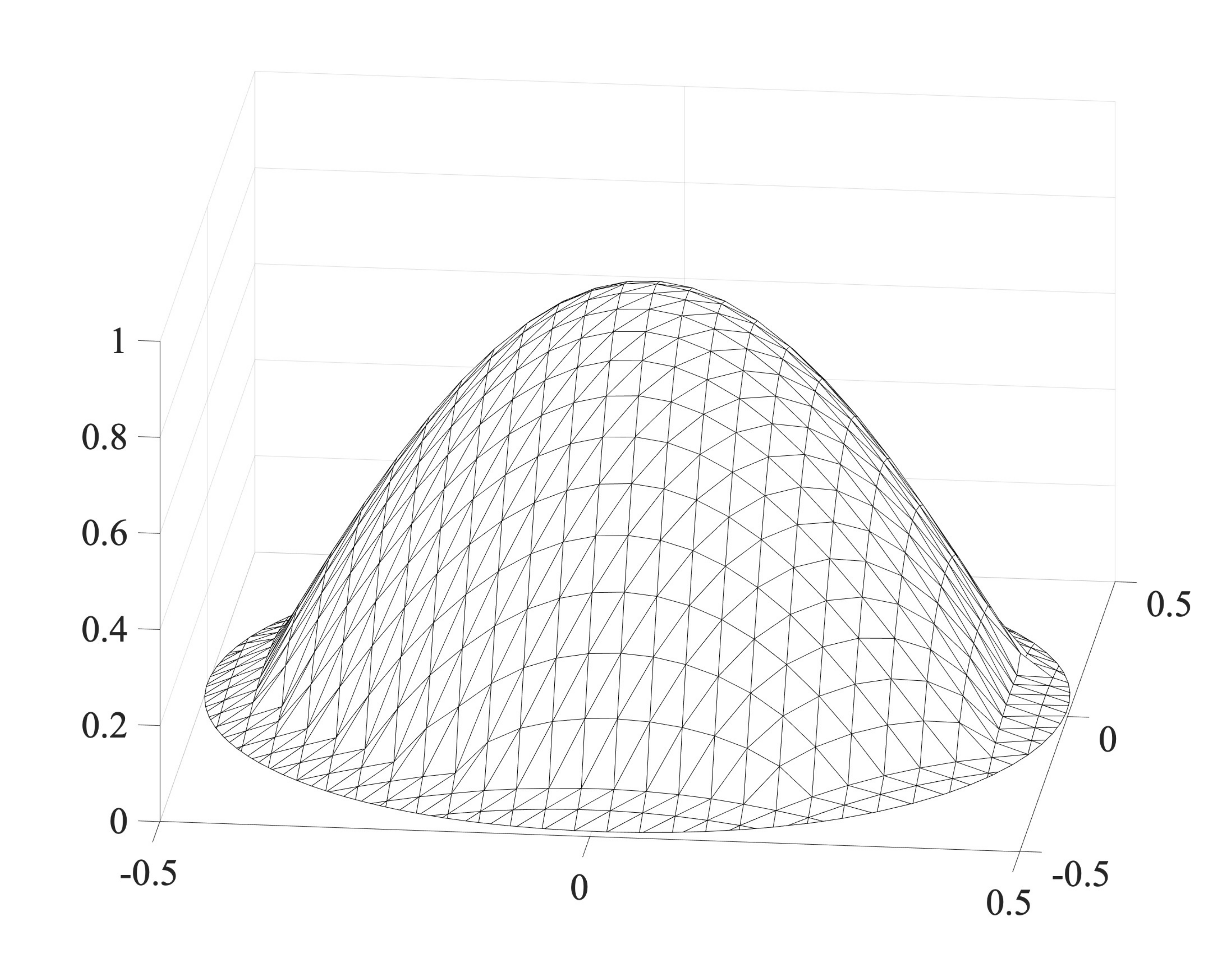}\includegraphics[scale=0.2]{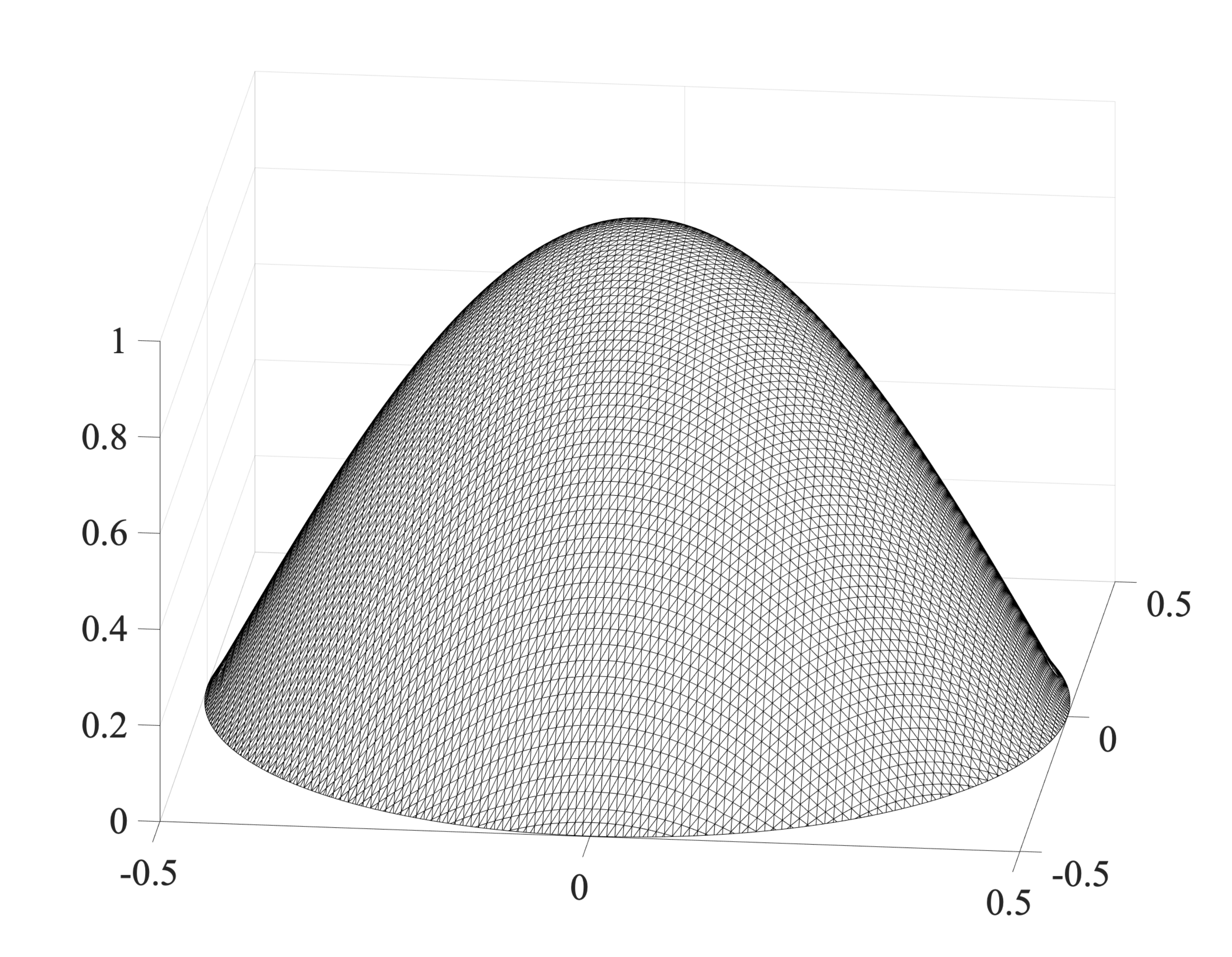}
	\end{center}
\caption{Elevation of the computed solution on a coarse and on a fine mesh for face penalty with $\tau=10^3$.\label{edge:elev}}
\end{figure}

\end{document}